\documentclass[11pt,thmsa,a4paper,reqno]{article}

\usepackage[mathscr]{euscript}
\usepackage{amssymb}
\usepackage{amsthm}
\usepackage{amsxtra}
\usepackage{bbm}
\usepackage{mathtools}
\usepackage{stmaryrd}

\usepackage{extarrows}
\usepackage{enumitem}
\usepackage{tikz-cd}
\usepackage{url}
\usepackage{booktabs}
\usepackage{graphicx}
\usepackage{resizegather}
\usepackage{arydshln}
\usepackage{floatrow}
\usepackage{soul}
\usepackage{relsize,etoolbox}
 \usepackage[all,cmtip,2cell]{xy}
\UseAllTwocells
\xyoption{2cell}

\numberwithin{equation}{section}

\usepackage[pdftex,
                paper=a4paper,
                portrait=true,
                textwidth=165mm,
                textheight=235mm,
                tmargin=2.5cm,
                marginratio=1:1]{geometry}

\makeatletter
\newtheorem*{rep@theorem}{\rep@title}
\newcommand{\newreptheorem}[2]{%
\newenvironment{rep#1}[1]{%
 \def\rep@title{#2 \ref{##1}}%
 \begin{rep@theorem}}%
 {\end{rep@theorem}}}
\makeatother

\newtheorem{theorem}{Theorem}[section]
\newreptheorem{theorem}{Theorem}
\newtheorem{proposition}[theorem]{Proposition}
\newtheorem{lemma}[theorem]{Lemma}

\newtheorem{corollary}[theorem]{Corollary}
\newreptheorem{corollary}{Corollary}

\theoremstyle{definition}

\theoremstyle{remark}
\newtheorem{remark}[theorem]{Remark}
 \newtheorem*{note}{Note added}

\newtheorem*{acknowledgements}{Acknowledgements}
\newtheorem*{NotationConventions}{Notation and conventions}

\newcommand{\RR}{\ensuremath{\mathbb{R}}}

\newcommand{\ZZ}{\ensuremath{\mathbb{Z}}}

\newcommand{\Ccal}{\EuScript{C}}
\newcommand{\Dcal}{\EuScript{D}}
\newcommand{\Ecal}{\EuScript{E}}

\newcommand{\ud}{\operatorname{d}\!}
\newcommand{\ui}{\operatorname{i}}

\newcommand{\End}{\operatorname{End}}
\newcommand{\Hom}{\operatorname{Hom}}
\newcommand{\Loc}{\operatorname{\mathbf{Loc}}}

\newcommand{\uZ}{\operatorname{Z}}
\newcommand{\Ob}{\operatorname{Ob}}

\newcommand{\DGVect}{\operatorname{\mathbf{DGVect}}}
\newcommand{\AFun}{\operatorname{\mathsf{A}_{\infty}\text{-}\mathbf{Fun}}}
\newcommand{\A}{\mathsf{A}}
\def\id{{\rm id}}
\def\uud{{\rm d}}

\newcommand*{\sbullet}{\raisebox{0.1ex}{\scalebox{0.6}{$\bullet$}}}

\DeclareMathAlphabet{\mathbfit}{OML}{cmm}{b}{it}

\begin{document}

\title{An $\A_{\infty}$ version of the Poincar\'e lemma}

\author{Camilo Arias Abad\footnote{Escuela de Matem\'{a}ticas, Universidad Nacional de Colombia Sede Medell\'{i}n, email: camiloariasabad@gmail.com}, Alexander Quintero V\'{e}lez\footnote{Escuela de Matem\'{a}ticas, Universidad Nacional de Colombia Sede Medell\'{i}n, email: aquinte2@unal.edu.co} and Sebasti\'an V\'elez V\'asquez\footnote{Escuela de Matem\'{a}ticas, Universidad Nacional de Colombia Sede Medell\'{i}n, email: svelezv@unal.edu.co}}

 \maketitle

\begin{abstract}
We prove a categorified version of the Poincar\'e lemma. The natural setting for our result is that of $\infty$-local systems. More precisely, we show that any smooth homotopy between maps $f$ and $g$ induces an $\A_\infty$-natural transformation between the corresponding pullback functors. This transformation is explicitly defined in terms of Chen's iterated integrals.
In particular, we show that a homotopy equivalence induces a quasi-equivalence on the DG categories of
$\infty$-local system.
\end{abstract}


\section{Introduction}
Higher versions of local systems on a smooth manifold has been considered in several recent works. Some of the references include  Block-Smith~\cite{Block-Smith2014}, Arias~Abad-Sch\"{a}tz~\cite{Abad-Schatz2016,Abad-Schatz2013}, Holstein~\cite{Holstein2015}, Malm~\cite{Malm2011}, Ben-Zvi-Nadler~\cite{Ben-Zvi-Nadler2012}, Brav-Dyckerhoff~\cite{Brav-Dyckerhoff2016} and Rivera-Zeinalian~\cite{Rivera-Zeinalian2018}. These references contain different points of view on such $\infty$-local systems, as they are now called. Crucially, each of the points of view can be used to define a DG category of $\infty$-local systems, and it has been shown that all the resulting DG categories are $\A_{\infty}$-quasi-equivalent \cite{Holstein2015,Abad-Schatz2013,Block-Smith2014}.

In this paper, we take the de Rham point of view that an $\infty$-local system on a manifold $M$ is a $\ZZ$-graded vector bundle equipped with a flat $\ZZ$-graded superconnection. We denote the corresponding DG category by $\Loc_{\infty}(M)$ and study its behaviour with respect to homotopies. It has been proved by Holstein~\cite{Holstein2015}, using a different but equivalent version of $\infty$-local systems, that the pullback by a homotopy equivalence induces a quasi-equivalence on local systems. We use the de Rham point of view on $\infty$-local systems to provide a more explicit version of this homotopy invariance property. The precise result is as follows.

\begin{reptheorem}{thm:4.5}
Let $M, N$ be smooth manifolds and let $h$ be a smooth homotopy between maps $f,g \colon M \to N$. Then there exists an $\A_{\infty}$-natural isomorphism $\mathrm{hol}\colon f^* \Rightarrow g^*$ between the pullback functors $f^*,g^* \colon \Loc_{\infty}(N) \to \Loc_{\infty}(M).$ Such an $\A_{\infty}$-natural isomorphism depends only on $h$ and is given explicitly by Chen's iterated integrals.
\end{reptheorem}

This result should be contrasted with the $\A_{\infty}$ version of de Rham's theorem due to Gugenheim~\cite{Gugenheim1977}, which plays a key role in the construction of the higher Riemann-Hilbert correspondence given in \cite{Block-Smith2014} and \cite{Abad-Schatz2016}. There are also some important corollaries that follow directly from it, including the following.

\begin{repcorollary}{cor:5.1}
If $f \colon M \to N$ is a smooth homotopy equivalence, then the pullback functor $f^{*} \colon \Loc_{\infty}(N) \to \Loc_{\infty}(M)$ is a quasi-equivalence.
\end{repcorollary}

\begin{repcorollary}{cor:5.2}
If $M$ is contractible, then $\Loc_{\infty}(M)$ is quasi-equivalent to $\DGVect_{\RR}$.
\end{repcorollary}

The latter should be thought of as a categorified version of the Poincar\'e lemma. It provides a local normal form for flat superconnections which we believe is of independent interest.

\begin{repcorollary} {cor:5.4}
For an arbitrary manifold $M$, any $\infty$-local system $(E,D)$ is locally isomorphic to a constant $\infty$-local system, that is, every point has an open neighbourhood $U$ in which $(E\vert_U,D\vert_U)$ is isomorphic to a constant $\infty$-local system.
\end{repcorollary}

We remark that an analogous result in the complex-analytic context have been established by Bondal-Rosly~\cite{Bondal-Rosly2011}. Some other related results are also considered in \cite{Demessie-Seamann2015}.

The paper is organized as follows.  Some preliminaries on DG categories, DG functors, $\A_{\infty}$-natural transformations and $\infty$-local systems are presented in \S\ref{sec:2}. In \S\ref{sec:3}, we describe some properties of Chen's iterated integrals that are used in our construction. The main section of the paper is \S\ref{sec:4}, which contains the proof of the homotopy invariance of the DG category of $\infty$-local systems. In \S\ref{sec:5}, we derive some corollaries of our main result, including the categorified version of the Poincar\'e lemma. Finally, Appendix \ref{app:A} contains some technical computations which are essential for the proof of the main result and Appendix \ref{app:B} a short compilation of the basic definitions of $\A_{\infty}$-categories.

\begin{note}
While we were writing this paper, the work \cite{CHL2018} appeared. Among other results, this paper discusses a higher version of the Riemann-Hilbert correspondence which generalizes and extends the results of \cite{Block-Smith2014}. In this reference the notion of $\infty$-local system is replaced by that of cohomologically locally constant DG sheaf with cohomology sheaves of finite rank, and their results are therefore closely related to ours.
\end{note}

\begin{NotationConventions}
If $K$ is a field and $V = \bigoplus_{k \in \ZZ} V^k$ is a $\ZZ$-graded $K$-vector space, we denote by $sV$ its suspension, that is, the $\ZZ$-graded $K$-vector space with grading defined by
$$
(sV)^{k} = V^{k+1}.
$$
Vector spaces and tensor products are defined over the real numbers unless otherwise stated.

If $E = \bigoplus_{k \in \ZZ} E^k$ is a $\ZZ$-graded vector bundle over a smooth manifold $
M$, we define
$$
\Omega^{\sbullet}(M,E)=\Gamma(\Lambda^{\sbullet} T^*M \otimes E).
$$
This space is graded by the total degree
$$
\Omega^{\sbullet}(M,E)^n = \bigoplus_{p+q=n} \Omega^{p}(M,E^q).
$$
Given an element $\omega \in \Omega^p(M,E^q)$ we will say that $\omega$ is of partial degree $p$.
\end{NotationConventions}

\begin{acknowledgements}
We would like to acknowledge the support of Colciencias through  their grant {\it Estructuras lineales en topolog\'ia y geometr\'ia}, with contract number FP44842-013-2018. Camilo and Sebasti\'an would also like to thank the Fields Institute in Toronto for their hospitality and support during the final stages of this project. Finally, we thank the referee for the helpful and constructive comments.
\end{acknowledgements}

\section{Preliminaries}\label{sec:2}
In this section we review some facts regarding DG categories and higher local systems that will be used throughout the paper. For a more thorough discussion on the topics treated here see for example
\cite{Keller2006,Positselski2011,Block2010,Block-Smith2014,AriasAbad-Crainic2012,Abad-Schatz2013}.

\subsection{DG categories, DG functors and $\A_{\infty}$-natural transformations}\label{subsec:2.1}
A \emph{DG category} (where DG stands for ``differential graded'') over a field $K$ is a $K$-linear category $\Ccal$ such that for every two objects $X$ and $Y$ the space of arrows $\Hom_{\Ccal}(X,Y)$ is equipped with a structure of a cochain complex of $K$-vector spaces, and for every three objects $X$, $Y$ and $Z$ the composition map \[\Hom_{\Ccal}(Y,Z) \otimes_K \Hom_{\Ccal}(X,Y) \to \Hom_{\Ccal}(X,Z)\] is a morphism of cochain complexes. Thus, by definition, \[\Hom_{\Ccal}(X,Y) = \bigoplus_{n \in \ZZ} \Hom_{\Ccal}^{n}(X,Y)\] is a $\ZZ$-graded $K$-vector space with a differential $d \colon \Hom_{\Ccal}^{n}(X,Y) \to \Hom_{\Ccal}^{n+1}(X,Y)$. The elements $f \in \Hom_{\Ccal}^{n}(X,Y)$ are called \emph{homogeneous of degree $n$}, and we write $\vert f \vert = n$. We shall denote the set of objects of $\Ccal$ by $\Ob \Ccal$.

The prototypical example of a DG category is the category of cochain complexes of $K$-vector spaces, which we denote by $\DGVect_K$. Its objects are cochain complexes of $K$-vector spaces and the morphism spaces $\Hom_{\DGVect_K}(X,Y)$ are endowed with the differential defined as
$$
d (f) = d_{Y} \circ f - (-1)^n f \circ d_{X},
$$
for any homogeneous element $f$ of degree $n$.

Let $\Ccal$ be a DG category and let $X \in \Ob \Ccal$. Given a closed morphism $f \in \Hom_{\Ccal}^0(Y,Z)$ we define $f_* \colon \Hom_{\Ccal}(X,Y) \to \Hom_{\Ccal}(X,Z)$ by $f_{*} (g) = f \circ g$ for $g \in \Hom_{\Ccal}(X,Y)$. It is not difficult to see that $f_*$ is a morphism of cochain complexes.  Similarly, if we define $f^* \colon  \Hom_{\Ccal}(Z,X) \to \Hom_{\Ccal}(Y,X)$ by $f^{*} (h) = h \circ f$ for $h \in \Hom_{\Ccal}(Z,X)$, then $f^*$ is a morphism of cochain complexes.

Given a DG category $\Ccal$ one can define an ordinary category $\mathbf{Ho}(\Ccal)$ by keeping the same set of objects and replacing each $\Hom$ complex by its $0$th cohomology. We call $\mathbf{Ho}(\Ccal)$ the \emph{homotopy category} of $\Ccal$.

If $\Ccal$ and $\Dcal$ are DG categories, a \emph{DG functor} $F \colon \Ccal \to \Dcal$ is an $K$-linear functor whose associated map for $X, Y \in \Ob \Ccal$,
$$
F_{X,Y} \colon \Hom_{\Ccal}(X,Y) \to \Hom_{\Dcal}(F(X),F(Y)),
$$
is a morphism of cochain complexes. Notice that any DG functor $F \colon \Ccal \to \Dcal$ induces an ordinary functor
$$
\mathbf{Ho}(F) \colon \mathbf{Ho}(\Ccal) \to \mathbf{Ho}(\Dcal)
$$
between the corresponding homotopy categories. A DG functor $F \colon \Ccal \to \Dcal$ is said to be \emph{quasi fully faithful} if for every pair of objects $X, Y \in \Ob \Ccal$ the morphism $F_{X,Y}$ is a quasi-isomorphism. Moreover, the DG functor $F $ is said to be \emph{quasi essentially surjective} if $\mathbf{Ho}(F)$ is essentially surjective. A DG functor which is both quasi fully faithful and quasi essentially surjective is called a \emph{quasi-equivalence}.

Let $F$ and $G$ be two functors between two DG categories $\Ccal$ and $\Dcal$. We want to define the notion of an $\A_{\infty}$-natural transformation $\lambda$ from $F$ to $G$. A thorough discussion of this in connection with the theory of $\A_{\infty}$-categories is presented in Appendix~\ref{app:B}. Here, we want to be explicit, and motivate the definition as follows. If $X, Y \in \Ob \Ccal$ and $f \in \Hom_{\Ccal}(X,Y)$, we can consider the standard diagram
$$
\xymatrix{ F(X) \ar[r]^-{\lambda(X)} \ar[d]_-{F(f)} & G(X) \ar[d]^-{G(f)} \\ F(Y) \ar[r]_-{\lambda(Y)} & G(Y)}
$$
Normally one would require this diagram to commute. However, in our case we will only require this diagram to be commutative up to homotopy. To make precise what this means we need one piece of notation.

Let $\Ccal$ be a DG category and let $X_0,\dots,X_n$ be a collection of objects of $\Ccal$. We endow
$$
s\! \Hom_{\Ccal}(X_{n-1},X_n) \otimes_K \cdots \otimes_K s\! \Hom_{\Ccal}(X_{0},X_1)
$$
with a differential $b$ defined as
\begin{align*}
b (f_{n-1} \otimes \cdots \otimes f_0) = &\sum_{i=0}^{n-1} (-1)^{\sum_{j=i+1}^{n-1}\vert f_{j} \vert  + n - i - 1}  f_{n-1} \otimes \cdots \otimes d f_i \otimes \cdots\otimes f_0 \\
& + \sum_{i=0}^{n-2} (-1)^{\sum_{j=i+2}^{n-1}\vert f_{j} \vert + n - i } f_{n-1} \otimes \cdots \otimes (f_{i+1} \circ f_i) \otimes \cdots\otimes f_0
\end{align*}
for homogeneous elements $f_0 \in  s\! \Hom_{\Ccal}(X_{0},X_1), \dots, f_{n-1} \in s\! \Hom_{\Ccal}(X_{n-1},X_n)$.  Here $d$ denotes indistinctly the differential in any of the spaces $\Hom_{\Ccal} (X_{i},X_{i+1})$. A direct calculation shows that indeed, $b^2 = 0$. Notice that the definition of $b$ resembles that of the differential of the Hochschild chain complex of a DG algebra.

Armed with this notation, the formal definition of an $\A_{\infty}$-natural transformation is given as follows. Let $\Ccal$ and $\Dcal$ be DG categories and let $F \colon \Ccal \to \Dcal$ and $G \colon \Ccal \to \Dcal$ be DG functors. An \emph{$\A_{\infty}$-natural transformation}  $\lambda \colon F \Rightarrow G$ is the datum of a closed morphism $\lambda_0(X) \in \Hom_{\Dcal}^0 (F(X),G(X))$ for each $X \in \Ob \Ccal$ and a collection of $K$-linear maps of degree $0$
$$
\lambda_n \colon s\! \Hom_{\Ccal}(X_{n-1},X_n) \otimes_K \cdots \otimes_K s\! \Hom_{\Ccal}(X_{0},X_1) \to \Hom_{\Dcal}(F(X_0),G(X_n))
$$
for every collection $X_0,\dots,X_n \in \Ob \Ccal$, such that for all composable chains of homogeneous morphisms $f_0 \in  s\! \Hom_{\Ccal}(X_{0},X_1), \dots, f_{n-1} \in s\! \Hom_{\Ccal}(X_{n-1},X_n)$ the relation
\begin{align*}
G(f_{n-1}) \circ \lambda_{n-1}(f_{n-2}\otimes \cdots \otimes f_0) &- (-1)^{\sum_{i=1}^{n-1}\vert f_i \vert -n+1} \lambda_{n-1}(f_{n-1} \otimes\cdots \otimes f_1) \circ F(f_0) \\
&\qquad\qquad  = \lambda \left( b(f_{n-1} \otimes \cdots \otimes f_0) \right) + d \left(\lambda_n (f_{n-1} \otimes \cdots \otimes f_0) \right)
\end{align*}
is satisfied for any $n \geq 1$. The $\lambda$ on the right denotes the direct sum of the various $\lambda_n$. For $n=1$ this yields the condition
$$
G(f_0) \circ \lambda_0(X_0) - \lambda_0(X_1) \circ F(f_0) = \lambda_1 \left(d(f_0) \right) + d \left(\lambda_1(f_0)\right).
$$
Since the map $\lambda_1 \colon s\!\Hom_{\Ccal}(X_{0},X_1) \to \Hom_{\Dcal}(F(X_{0}),G(X_1))$ has degree $-1$ when considered as a map defined over $\Hom_{\Ccal}(X_{0},X_1)$, this implies that the diagram
$$
\xymatrix{ F(X_0) \ar[r]^-{\lambda_0(X_0)} \ar[d]_-{F(f_0)} & G(X_0) \ar[d]^-{G(f_0)} \\ F(X_1) \ar[r]_-{\lambda_0(X_1)} & G(X_1)}
$$
commutes up to a homotopy given by $\lambda_1$. More generally, for $n \geq 2$, one may say that $\lambda_{n-1}$ ``commutes'' with $G(f_{n-1})$ and $F(f_0)$ up to a homotopy given by $\lambda_{n}$.

As usual, $\A_{\infty}$-natural transformations can be composed: if $F \colon \Ccal \to \Dcal$, $G \colon \Ccal \to \Dcal$ and $H \colon \Ccal \to \Dcal$ are three DG functors from the DG category $\Ccal$ to the DG category $\Dcal$, and $\lambda \colon F \Rightarrow G$ and $\mu \colon G \Rightarrow H$ are two $\A_{\infty}$-natural transformations, then the formula
$$
(\mu \circ \lambda)_n = \sum_{i=0}^n \mu_i \circ \lambda_{n-i}
$$
defines a new $\A_{\infty}$-natural transformation $\mu \circ \lambda \colon F \Rightarrow H$.  An \emph{$\A_{\infty}$-natural isomorphism} between functors from $\Ccal$ to $\Dcal$ is an $\A_\infty$-natural transformation $\lambda$
such that $\lambda_0(X)$ is an isomorphism for all $X \in \Ob \Ccal$.

We close this section with the following observation.

\begin{lemma}\label{lem:2.1}
Let $\Ccal$ and $\Dcal$ be DG categories and let $F \colon \Ccal \to \Dcal$ be a DG functor. Suppose that there is a DG functor $G \colon \Dcal \to \Ccal$ together with $\A_{\infty}$-natural isomorphisms $\lambda \colon G \circ F \Rightarrow \id_{\Ccal}$ and $\mu \colon F \circ G \Rightarrow \id_{\Dcal}$. Then $F$ is a quasi-equivalence.
\end{lemma}

\begin{proof}
Let us first prove that $F$ is quasi essentially surjective. Given any object $Y \in \Ob \Dcal$ the morphism $\mu_0(Y) \in \Hom_{\Dcal}^0(F(G(Y)),Y)$ is an isomorphism. In particular it descends to an isomorphism in the homotopy category. Let us show that $F$ is quasi fully faithful. By definition, if  $f \in s\!\Hom_{\Ccal}(X,Y)$ is a homogenous element, the $\A_{\infty}$-natural isomorphism $\lambda \colon G \circ F \Rightarrow \id_{\Ccal}$ produces the relation
$$
f \circ \lambda_0(X) - \lambda_0(Y) \circ G(F(f)) = \lambda_1 \left( d(f)\right) + d \left( \lambda_1(f)\right).
$$
This, in turn, may be written as
$$
\lambda_0(X)^* \circ \id_{\Ccal} - \lambda_0(Y)_* \circ (G\circ F) = \lambda_1 \circ d + d \circ \lambda_1,
$$
for any pair of objects $X,Y \in \Ob \Ccal$, where $\lambda_0(X)^*$ is the pullback of $\lambda_0(X)$ by $X$ and $\lambda_0(Y)_*$ is the pushforward of $\lambda_0(Y)$ by $G(F(Y))$. Hence, the morphisms of cochain complexes $\lambda_0(X)^* \circ \id_{\Ccal}$ and $\lambda_0(Y)_* \circ (G\circ F)$ are homotopic and, therefore, they induce the same morphism in cohomology. But clearly $\lambda_0(X)^*$, $\lambda_0(Y)_*$ and $\id_{\Ccal}$ induce isomorphisms in cohomology, and thus so does $G \circ F$. It follows that $G \circ F$ is a quasi-isomorphism. By an entirely analogous argument, using the $\A_{\infty}$-natural isomorphism $\mu \colon F \circ G \Rightarrow \id_{\Dcal}$, one may prove that $F \circ G$ is a quasi-isomorphism. The desired implication follows at once.
\end{proof}

\subsection{$\infty$-Local systems}\label{subsec:2.2}
Let $E = \bigoplus_{k \in \ZZ} E^k$ be a $\ZZ$-graded vector bundle over a manifold $M$. We consider the space of $E$-valued differential forms $\Omega^{\sbullet}(M,E)$ to be $\ZZ$-graded with respect to the total degree. A \emph{$\ZZ$-graded superconnection} on $E$ is an operator $D \colon \Omega^{\sbullet} (M,E) \to \Omega^{\sbullet}(M,E)$ of degree $1$ which satisfies the Leibniz rule
$$
D (\sigma \wedge \omega)= \ud \sigma \wedge \omega + (-1)^k \sigma \wedge D \omega,
$$
for all $\sigma \in \Omega^k(M)$ and $\omega \in  \Omega^{\sbullet} (M,E)$. The \emph{curvature} of $D$ is the operator $D^2$. This is an $\Omega^{\sbullet}(M)$-linear operator on $\Omega^{\sbullet} (M,E)$ of degree $2$ which is given by multiplication by an element of $\Omega^{\sbullet} (M,\End(E))$. If $D^2=0$, then we say that $D$ is a \emph{flat} $\ZZ$-graded superconnection. By an \emph{$\infty$-local system} on $M$ we mean a $\ZZ$-graded vector bundle $E$ equipped with a flat $\ZZ$-graded superconnection $D$. We will denote such an $\infty$-local system by $(E,D)$.

As a simple example, consider a trivial vector bundle $M \times V$ with fiber $V = \bigoplus_{k \in \ZZ} V^k$ a $\ZZ$-graded vector space. It is a easy matter to verify that the de Rham differential $\uud \colon \Omega^{\sbullet}(M) \to \Omega^{\sbullet}(M)$ can be extended to an operator $\uud \colon \Omega^{\sbullet}(M, V) \to \Omega^{\sbullet}(M, V)$ of degree $1$. With the identification $\Omega^{\sbullet}(M, V) = \Omega^{\sbullet}(M, M \times V)$, one can readily show that $(M \times V, \uud)$ defines an $\infty$-local system on $M$. We will refer to it as a \emph{constant $\infty$-local system} on $M$.

For convenience of reference, we make the following observation. Suppose that $(E,D)$ is an $\infty$-local system on $M$. The Leibniz rule implies that $D$ is completely determined by its restriction to $\Omega^0(M,E)$. Then we may decompose
$$
D = \sum_{k \geq 0} D_k,
$$
where $D_k$ is of partial degree $k$ with respect to the $\ZZ$-grading on $\Omega^{\sbullet}(M)$. It is clear that each $D_k$ for $k \neq 1$ is $\Omega^{\sbullet}(M)$-linear and therefore it is given by multiplication by an element $-\alpha_k \in \Omega^k(M,\End(E)^{1-k})$ (the minus sign is only a matter of convention). On the contrary, $D_1$ satisfies the Leibniz rule on each of the vector bundles $E^k$, so it must be of the form $\ud_{\nabla}$, where $\nabla$ is an ordinary connection on $E$ which preserves the $\ZZ$-grading. We can thus write
$$
D = \ud_{\nabla} - \alpha_0 - \alpha_2 - \alpha_3 - \cdots.
$$
From this formula, it is straightforward to check that the flatness condition becomes equivalent to
\begin{align*}
\alpha_0^2 &= 0, \\
\ud_{\nabla} \alpha_0 &=0, \\
[\alpha_0,\alpha_2] + F_{\nabla} &= 0, \\
[\alpha_0,\alpha_{n+1}] + \ud_{\nabla}\alpha_n + \sum_{k=2}^{n-1} \alpha_k \wedge \alpha_{n+1-k} &= 0, \,\, n \geq 2,
\end{align*}
where $F_{\nabla}$ is the curvature of the connection $\nabla$. The first identity implies that we have a cochain complex of vector bundles with differential $\alpha_0$. The second equation express the fact that $\alpha_0$ is covariantly constant with respect to the connection $\nabla$. The third equation indicates that the connection $\nabla$ fails to be flat up to terms involving the homotopy $\alpha_2$ and the differential $\alpha_0$.

Now let us assume that $E$ is trivialized over $M$. This means that $E = M \times V$ for some $\ZZ$-graded vector space $V = \bigoplus_{k \in \ZZ} V^{k}$. In this case, we have $\alpha_k \in \Omega^k(M, \End(V)^{1-k})$ for $k \neq 1$. Moreover, we can write $\ud_{\nabla} = \uud - \alpha_1$ for some $\alpha_1 \in \Omega^1(M, \End(V)^{0})$. Thus, the $\ZZ$-graded superconnection $D$ may be expressed as $D = \uud - \alpha$, where $\alpha \in \Omega^{\sbullet}(M,\End(V))$ is the homogenous element of total degree $1$ defined by $\alpha = \sum_{k \geq 0} \alpha_k$. In addition, a straightforward calculation gives
$$
D^2  = \ud \alpha - \alpha \wedge \alpha.
$$
Consequently, the totality of equations of the flatness condition is equivalent to the single statement that $\alpha$ satisfies
$$
\ud \alpha - \alpha \wedge \alpha = 0.
$$
This is known as the \emph{Maurer-Cartan equation}.

Suppose we have another trivialization of $E$ over $M$ such that $E = M \times W$ for some $\ZZ$-graded vector space $W = \bigoplus_{k \in \ZZ} W^k$ and $D = \uud - \beta$ for some homogenous element $\beta \in \Omega^{\sbullet}(M,\End(W))$ of total degree $1$ satisfying the Maurer-Cartan equation. Then, we have a transition isomorphism between the two trivializations, which is realized by a linear isomorphism $g \colon \Omega^0(M,V) \to \Omega^0(M,W)$ that commutes with the operators $\uud - \alpha$ and $\uud - \beta$. If we think of $g$ as an element of $\Omega^0(M,\Hom(V,W))$, the latter condition is equivalent to the requirement that
$$
\alpha = g^{-1} \beta g - g^{-1} \ud g.
$$
The change from $\beta$ to $\alpha$ given in this equation goes by the name of a ``gauge transformation''.

For a pair of $\infty$-local systems $(E,D)$ and $(E',D')$ there is a natural notion of a morphism from $(E,D)$ to $(E',D')$. Namely, such a morphism is a degree $0$ linear map $\Phi \colon \Omega^{\sbullet}(M,E) \to \Omega^{\sbullet}(M,E')$ which is $\Omega^{\sbullet}(M)$-linear and commutes with the $\ZZ$-graded superconnections $D$ and $D'$. If both $(E,D)$ and $(E',D')$ are trivialized over $M$ in such a way that $E = M \times V$ and $E'=M \times V'$ for some $\ZZ$-graded vector spaces $V$ and $V'$, and $D = \uud - \alpha$ and $D'=\uud - \alpha'$ for some homogeneous elements $\alpha \in \Omega^{\sbullet}(M,\End(V))$ and $\alpha' \in \Omega^{\sbullet}(M,\End(V'))$ of total degree $1$ satisfying the Maurer-Cartan equation, then this condition is
$$
(\uud - \alpha') \circ \Phi = \Phi \circ (\uud - \alpha),
$$
or, interpreting $\Phi$ as an element of $\Omega^{\sbullet}(M,\Hom(V,V'))$,
$$
\ud \Phi = \alpha' \wedge \Phi - \Phi \wedge \alpha.
$$
A morphism from $(E,D)$ to $(E',D')$ is called an \emph{isomorphism} if the underlying map $\Phi \colon \Omega^{\sbullet}(M,E) \to  \Omega^{\sbullet}(M,E')$ is an isomorphism. Two $\infty$-local systems $(E,D)$ and $(E',D')$ are said to be \emph{isomorphic} if there is an isomorphism from $(E,D)$ to $(E',D')$.

As mentioned in the introduction, all $\infty$-local systems on a manifold $M$ can be naturally organized into a DG category, which we denote by $\Loc_{\infty}(M)$. Its objects are, of course, $\infty$-local systems $(E,D)$ on $M$. Given two $\infty$-local systems $(E,D)$ and $(E',D')$ we define the space of morphisms to be the $\ZZ$-graded vector space $\Omega^{\sbullet}(M,\Hom(E,E'))$ with the differential $\partial_{D,D'}$ acting as
$$
\partial_{D,D'} \omega = D' \wedge \omega - (-1)^{k} \omega \wedge D,
$$
for any homogenous element $\omega$ of degree $k$. If $(E,D)$ and $(E',D')$ are trivialized over $M$ as in the previous paragraph, then $\partial_{D,D'}$ may be expressed by
$$
\partial_{D,D'} \omega = \ud \omega - \alpha' \wedge \omega + (-1)^{k} \omega \wedge \alpha.
$$
Notice that what we call a morphism from $(E,D)$ to $(E',D')$ is simply a closed element of $\Omega^{\sbullet}(M,\Hom(E,E'))$ of degree $0$.

Finally, to close this section, let us briefly recall the pullback operation of $\infty$-local systems. For a smooth map $f \colon M \to N$ between two manifolds $M$ and $N$, there is a DG functor $f^{*} \colon \Loc_{\infty}(N) \to \Loc_{\infty}(M)$ which sends $E$ with structure superconnection
$$
D = \ud_{\nabla} - \alpha_0 - \alpha_2 - \alpha_3 - \cdots,
$$
to $f^* E$ endowed with
$$
f^* D = \ud_{f^*\nabla} - f^* \alpha_0 - f^* \alpha_2 - f^* \alpha_3 - \cdots,
$$
where $f^*E$ is the pullback of $E$ and $f^* \nabla$ is the pullback connection on $f^{*} E$. One can easily check that $(f^*E,f^*D)$ is indeed an $\infty$-local system on $M$, so the DG functor $f^*$ is well defined. We refer to it as the \emph{pullback functor} induced by $f$.

\section{Some generalities on Chen's iterated integrals}\label{sec:3}

In this section, we state some properties of the type of iterated integrals that appear in our study. We make use of the notation and conventions of \S\ref{subsec:2.1} and \S\ref{subsec:2.2}.

Let $M$ be a smooth manifold and let $V = \bigoplus_{k \in \ZZ} V^k$ be a $\ZZ$-graded vector space. We denote by $\iota_s \colon M \to M \times [1,0]$ the inclusion at height $s$ given by $\iota_s(x) = (x,s)$. For an element $\omega$ of $\Omega^{\sbullet} (M \times [0,1], \End (V))$ we define for all $t \in [0,1]$ the following elements of $\Omega^{\sbullet}(M,\End(V))$,
\begin{align*}
 \Phi_0^{\omega}(t) &= \id_{V}, \\
 \Phi_1^{\omega}(t) &= \int_{0}^{t} \iota_{s_1}^{*} \ui_{\frac{\partial}{\partial s_1}} \omega \, \ud s_1, \\
 \Phi_n^{\omega}(t) &= \int_{0}^{t}  \int_{0}^{s_1} \cdots \int_{0}^{s_{n-1}} \iota_{s_1}^{*} \ui_{\frac{\partial}{\partial s_1}} \omega\wedge \iota_{s_2}^{*} \ui_{\frac{\partial}{\partial s_2}} \omega \wedge \cdots \wedge \iota_{s_n}^{*} \ui_{\frac{\partial}{\partial s_n}} \omega \, \ud s_n \cdots \ud s_2 \ud s_1, \,\, n\geq 2,
\end{align*}
where in the definition of $ \Phi_n^{\omega}(t)$, $t \geq s_1 \geq \cdots \geq s_n \geq 0$. The following lemma is easy to verify.

\begin{lemma}\label{lem:3.1}
The series
\begin{equation*}
\Phi^{\omega}(t) = \sum_{n=0}^{\infty}  \Phi_n^{\omega}(t)
\end{equation*}
converges for all $t \in [0,1]$ and defines a smooth map from $[0,1]$ to $\Omega^{\sbullet} (M,\End(V))$.
\end{lemma}

We call $\Phi^{\omega}(t)$ the \emph{iterated integral} of the element $\omega$. Compare this with the definitions of \cite{Igusa2009} and the more classical reference \cite{Chen1977}.

These iterated integrals define solutions to differential equations we are interested in. More precisely, we have the following.

\begin{proposition}\label{prop:3.2}
Consider an element $\omega$ of $\Omega^{\sbullet}(M \times [0,1], \End(V))$ and the smooth map from $[0,1]$ to $\Omega^{\sbullet}(M, \End(V))$ given by $t \mapsto \Phi^{\omega}(t)$. Then
\begin{align*}
\begin{split}
 \frac{\ud \Phi^{\omega}(t)}{\ud t} &= \iota_{t}^* \ui_{\frac{\partial}{\partial t}} \omega \wedge \Phi^{\omega}(t), \\
\Phi^{\omega}(0) &=\id_V,
\end{split}
\end{align*}
for all $t \in [0,1]$.
\end{proposition}

\begin{proof}
It is clear that $\Phi^{\omega}(0) = \id_V$. Let us calculate the derivative with respect to $t$. According to Lemma~\ref{lem:3.1}, we have
\begin{align*}
\Phi^{\omega}(t) = &\id_V +  \int_{0}^{t} \iota_{s_1}^{*} \ui_{\frac{\partial}{\partial s_1}} \omega \, \ud s_1 +  \int_{0}^{t}  \int_{0}^{s_1}  \iota_{s_1}^{*} \ui_{\frac{\partial}{\partial s_1}} \omega\wedge \iota_{s_2}^{*} \ui_{\frac{\partial}{\partial s_2}} \omega \, \ud s_2 \ud s_1 \\
& +  \int_{0}^{t}  \int_{0}^{s_1} \int_{0}^{s_2}  \iota_{s_1}^{*} \ui_{\frac{\partial}{\partial s_1}} \omega\wedge \iota_{s_2}^{*} \ui_{\frac{\partial}{\partial s_2}} \omega \wedge  \iota_{s_3}^{*} \ui_{\frac{\partial}{\partial s_3}} \omega \, \ud s_3 \ud s_2 \ud s_1 + \cdots
\end{align*}
Therefore,
\begin{align*}
 \frac{\ud \Phi^{\omega}(t)}{\ud t} &= \iota_{t}^{*} \ui_{\frac{\partial}{\partial t}} \omega  + \iota_{t}^{*} \ui_{\frac{\partial}{\partial t}} \omega \wedge \int_{0}^{t} \iota_{s_2}^{*} \ui_{\frac{\partial}{\partial s_2}} \omega \, \ud s_2  \\
& \phantom{=} \, +  \iota_{t}^{*} \ui_{\frac{\partial}{\partial t}} \omega \wedge \int_{0}^{t}  \int_{0}^{s_2}  \iota_{s_2}^{*} \ui_{\frac{\partial}{\partial s_2}} \omega\wedge \iota_{s_3}^{*} \ui_{\frac{\partial}{\partial s_3}} \omega \, \ud s_3 \ud s_2 + \cdots \\
&=  \iota_{t}^{*} \ui_{\frac{\partial}{\partial t}} \omega \wedge \left( \id_V + \int_{0}^{t} \iota_{s_1}^{*} \ui_{\frac{\partial}{\partial s_1}} \omega \, \ud s_1 +  \int_{0}^{t}  \int_{0}^{s_1}  \iota_{s_1}^{*} \ui_{\frac{\partial}{\partial s_1}} \omega\wedge \iota_{s_2}^{*} \ui_{\frac{\partial}{\partial s_2}} \omega \, \ud s_2 \ud s_1 + \cdots\right) \\
&= \iota_{t}^{*} \ui_{\frac{\partial}{\partial t}} \omega \wedge \Phi^{\omega}(t),
\end{align*}
as asserted.
\end{proof}

Another important result that we need is the following gauge invariance property of iterated integrals.

\begin{proposition}\label{prop:3.3}
Let $W = \bigoplus_{k \in \ZZ} W^{k}$ be another $\ZZ$-graded vector space and let $\omega \in \Omega^{\sbullet}(M \times [0,1], \End(V))$ and $\eta \in \Omega^{\sbullet}(M \times [0,1], \End(W))$. Suppose that there is an invertible element $g \in \Omega^0 (M \times [0,1], \Hom(V,W))$ such that $\omega = g^{-1} \eta g - g^{-1} \ud g$. Then
\begin{equation*}
\Phi^{\omega}(t) = (\iota_t^* g)^{-1} \Phi^{\eta}(t) \iota_0^* g,
\end{equation*}
for all $t \in [0,1]$.
\end{proposition}

\begin{proof}
The strategy of the proof is to show that the right hand side of the given formula satisfies the initial value problem of Proposition~\ref{prop:3.2}. Uniqueness will then tell us that both sides of the formula are equal. To start with, it is clear that $(\iota_t^* g)^{-1} \Phi^{\eta}(t) \iota_0^* g$ satisfies the initial condition. Let us calculate its derivative with respect to $t$. By Proposition~\ref{prop:3.2}, we have
\begin{align*}
\frac{\uud}{\ud t} \left[ (\iota_t^* g)^{-1} \Phi^{\eta}(t) \iota_0^* g \right] &= \frac{\uud}{\ud t} (\iota_t^* g)^{-1} \Phi^{\eta}(t) \iota_0^* g + (\iota_t^* g)^{-1}\frac{\ud \Phi^{\eta}(t)}{\ud t}  \iota_0^* g \\
&=\frac{\uud}{\ud t} (\iota_t^* g)^{-1} \Phi^{\eta}(t) \iota_0^* g + (\iota_t^* g)^{-1}\left( \iota_{t}^{*} \ui_{\frac{\partial}{\partial t}} \eta \wedge \Phi^{\eta}(t)\right) \iota_0^* g \\
&= \frac{\uud}{\ud t} (\iota_t^* g)^{-1} \Phi^{\eta}(t) \iota_0^* g + \left[(\iota_t^* g)^{-1} \iota_{t}^{*} \ui_{\frac{\partial}{\partial t}} \eta \iota_t^* g  \right] \wedge \left[ (\iota_t^* g)^{-1}  \Phi^{\eta}(t) \iota_0^* g\right]
\end{align*}
On the other hand, a straightforward calculation shows that
$$
\frac{\uud}{\ud t} (\iota_t^* g)^{-1} = - (\iota_t^* g)^{-1}  \frac{\uud}{\ud t} (\iota_t^* g) (\iota_t^* g)^{-1}.
$$
Furthermore, using the assumption, we have
$$
(\iota_t^* g)^{-1} \iota_{t}^{*} \ui_{\frac{\partial}{\partial t}} \eta \iota_t^* g =\iota_{t}^{*} \ui_{\frac{\partial}{\partial t}} \omega + (\iota_t^* g)^{-1} \frac{\uud}{\ud t} (\iota_t^* g).
$$
Substituting into our previous equation we get
\begin{align*}
\frac{\uud}{\ud t} \left[ (\iota_t^* g)^{-1} \Phi^{\eta}(t) \iota_0^* g \right]  &=  - (\iota_t^* g)^{-1}  \frac{\uud}{\ud t} (\iota_t^* g) (\iota_t^* g)^{-1}\Phi^{\eta}(t) \iota_0^* g  + \iota_{t}^{*} \ui_{\frac{\partial}{\partial t}} \omega \wedge \left[ (\iota_t^* g)^{-1}  \Phi^{\eta}(t) \iota_0^* g\right] \\
&\phantom{=}\,\,  + (\iota_t^* g)^{-1}  \frac{\uud}{\ud t} (\iota_t^* g)  (\iota_t^* g)^{-1}  \Phi^{\eta}(t) \iota_0^* g \\
&= \iota_{t}^{*} \ui_{\frac{\partial}{\partial t}} \omega \wedge \left[ (\iota_t^* g)^{-1}  \Phi^{\eta}(t) \iota_0^* g\right],
\end{align*}
as required.
\end{proof}

It will be convenient for our purposes to give an equivalent alternative expression for $\Phi^{\omega}(t)$ in the special case in which $\omega$ is homogeneous. To do this we need some notation. For each $t \in [0,1]$, we write $\Delta_n(t)$ for the $n$-simplex of width $t$. The geometric realization of $\Delta_n(t)$ that we take is
$$
\Delta_n(t) = \{(s_1,\dots, s_n) \in \RR^n \mid  t \geq s_1 \geq \cdots \geq s_n \geq 0 \}.
$$
For any $i =1,\dots, n$, we also denote by $\pi_{i} \colon M \times \Delta_n(t) \to M \times [0,t]$ the natural projection defined by $
\pi_{i} (x, (s_1,\dots,s_n)) = (x, s_i)$ for all $x \in M$ and $(s_1,\dots,s_n) \in \Delta_n(t)$. Instead of $\Delta_n(1)$, we will simply write $\Delta_n$.

With this notation, we have the following.

\begin{proposition} \label{prop:3.4}
If $\omega$ is a homogeneous element of $\Omega^{\sbullet} (M \times [0,1], \End (V))$, then
\begin{equation*}
\Phi^{\omega}(t) = \id_V + \sum_{n=1}^{\infty}(-1)^{\varepsilon(n)\vert \omega \vert} \int_{\Delta_{n}(t)} \pi_{1}^* \omega \wedge \pi_{2}^* \omega \wedge \cdots \wedge \pi_{n}^* \omega,
\end{equation*}
where $\varepsilon(n)= \sum_{i=1}^{n-1}(n-i)$.
\end{proposition}

\begin{proof}
For each $(s_1,\dots,s_n) \in \Delta_n(t)$, let $\iota_{(s_1,\dots,s_n)} \colon M \to M \times \Delta_n(t)$ denote the natural inclusion given by $\iota_{(s_1,\dots,s_n)}(x)=(x, (s_1,\dots,s_n))$. Then, by definition,
\begin{align*}
& \int_{\Delta_{n}(t)} \pi_{1}^* \omega \wedge \pi_{2}^* \omega \wedge \cdots \wedge \pi_{n}^* \omega \\
  &\quad = \int_{0}^{t}  \int_{0}^{s_1} \cdots \int_{0}^{s_{n-1}} \iota_{(s_1,\dots,s_n)}^* \ui_{\frac{\partial}{\partial s_1}}  \ui_{\frac{\partial}{\partial s_2}} \cdots \ui_{\frac{\partial}{\partial s_n}} \left(\pi_{1}^* \omega \wedge \pi_{2}^* \omega \wedge \cdots \wedge \pi_{n}^* \omega \right) \, \ud s_n \cdots \ud s_2 \ud s_1.
\end{align*}
But $\ui_{\frac{\partial}{\partial s_i}} \pi_{j}^* \omega$ is equal to $\ui_{\frac{\partial}{\partial s_i}} \pi_{i}^* \omega$ if $i=j$ and is equal to $0$ otherwise. Moreover, since $\pi_{i} \circ \iota_{(s_1,\dots,s_n)} = \iota_{s_i}$, we also have that $\iota_{(s_1,\dots,s_n)}^* \ui_{\frac{\partial}{\partial s_i}} \pi_{i}^* \omega= \iota_{s_i}^*  \ui_{\frac{\partial}{\partial s_i}} \omega$. Therefore the integrand above becomes
\begin{align*}
&\iota_{(s_1,\dots,s_n)}^* \ui_{\frac{\partial}{\partial s_1}}  \ui_{\frac{\partial}{\partial s_2}} \cdots \ui_{\frac{\partial}{\partial s_n}} \left(\pi_{1}^* \omega \wedge \pi_{2}^* \omega \wedge \cdots \wedge \pi_{n}^* \omega \right) \\
&\qquad  = (-1)^{(n-1)\vert \omega \vert + (n-2)\vert \omega \vert + \cdots + \vert \omega \vert} \iota_{s_1}^*  \ui_{\frac{\partial}{\partial s_1}} \omega \wedge \iota_{s_2}^*  \ui_{\frac{\partial}{\partial s_2}} \omega  \wedge \cdots \wedge \iota_{s_n}^*  \ui_{\frac{\partial}{\partial s_n}} \omega.
\end{align*}
The desired conclusion now follows from the definition of $\Phi^{\omega}(t)$.
\end{proof}

We now  provide a simple observation which will be useful in its own right.

\begin{lemma}\label{lem:3.5}
If $\omega_1,\dots,\omega_n$ is a collection of homogeneous elements in $\Omega^{\sbullet}(M \times [0,1],\End(V))$, then
\begin{align*}
\ud \int_{\Delta_{n}(t)} & \, \pi_{1}^* \omega_1 \wedge \cdots \wedge \pi_{n}^* \omega_n \\
=&  \sum_{i=1}^{n} (-1)^{n + \sum_{j=1}^{i-1}\vert \omega_j \vert} \int_{\Delta_{n}(t)} \pi_{1}^* \omega_1  \wedge \cdots \wedge \pi_{i-1}^* \omega_{i-1} \wedge  \pi_{i}^*\ud \omega_{i}  \wedge \pi_{i+1}^* \omega_{i+1} \wedge \cdots \wedge \pi_{n}^* \omega_n \\
&+  \sum_{i=1}^{n-1}(-1)^{i} \int_{\Delta_{n-1}(t)} \pi_{1}^* \omega_1  \wedge \cdots \wedge \pi_{i-1}^* \omega_{i-1} \wedge  \pi_{i}^* (\omega_{i} \wedge \omega_{i+1})  \wedge \pi_{i+1}^* \omega_{i+2} \wedge \cdots \wedge \pi_{n-1}^* \omega_n \\
&+  (-1)^{n} \left( \int_{\Delta_{n-1}(t)} \pi_{1}^* \omega_1  \wedge \cdots \wedge \pi_{n-1}^* \omega_{n-1} \right) \wedge \iota_0^* \omega_n \\
&+  (-1)^{(n-1) \vert \omega_1\vert} \iota_t^{*} \omega_1 \wedge \left( \int_{\Delta_{n-1}(t)} \pi_{1}^* \omega_2  \wedge \cdots \wedge \pi_{n-1}^* \omega_{n} \right).
\end{align*}
\end{lemma}

\begin{proof}
It is straightforward to check that
$$
 \int_{\Delta_{n}(t)} \pi_{1}^* \omega_1 \wedge \cdots \wedge \pi_{n}^* \omega_n = (-1)^{(n-1)\vert \omega_1 \vert} \int_0^{t} \iota_s^* \ui_{\frac{\partial}{\partial s}} \omega_1 \wedge \left(  \int_{\Delta_{n-1}(s)} \pi_{1}^* \omega_2 \wedge \cdots \wedge \pi_{n-1}^* \omega_n\right) \, \ud s,
$$
so the result follows by induction on $n$.
\end{proof}

As a direct consequence of Lemma~\ref{lem:3.5}, we can prove the following formula, which will be used in what follows.
\begin{proposition}\label{prop:3.6}
If $\omega$ is a homogeneous element of $\Omega^{\sbullet} (M \times [0,1], \End (V))$, then
\begin{align*}
 \ud \Phi^{\omega}(t)  = &\sum_{n=1}^{\infty} \sum_{i=1}^{n} (-1)^{n + (i -1 + \varepsilon(n) )\vert \omega \vert} \int_{\Delta_{n}(t)} \pi_{1}^* \omega  \wedge \cdots \wedge \pi_{i-1}^* \omega \wedge  \pi_{i}^*\ud \omega  \wedge \pi_{i+1}^* \omega \wedge \cdots \wedge \pi_{n}^* \omega \\
&+ \sum_{n=2}^{\infty} \sum_{i=1}^{n-1}(-1)^{i + \varepsilon(n)\vert \omega \vert} \int_{\Delta_{n-1}(t)} \pi_{1}^* \omega  \wedge \cdots \wedge \pi_{i-1}^* \omega \wedge  \pi_{i}^* (\omega \wedge \omega)  \wedge \pi_{i+1}^* \omega \wedge \cdots \wedge \pi_{n}^* \omega \\
&+ \sum_{n=1}^{\infty} (-1)^{n + \varepsilon(n) \vert \omega \vert} \left( \int_{\Delta_{n-1}(t)} \pi_{1}^* \omega  \wedge \cdots \wedge \pi_{n-1}^* \omega \right) \wedge \iota_0^* \omega \\
&+ \sum_{n=1}^{\infty} (-1)^{(n-1 + \varepsilon(n)) \vert \omega\vert} \iota_t^{*} \omega \wedge \left( \int_{\Delta_{n-1}(t)} \pi_{1}^* \omega  \wedge \cdots \wedge \pi_{n-1}^* \omega \right).
\end{align*}
\end{proposition}

\section{The $\A_\infty$-natural transformation and homotopy invariance}\label{sec:4}

In this section we prove the main result of the paper, which is the construction of an $\A_\infty$-isomophism between the pullback functors associated to homotopic maps. This may be thought of as a categorified version of what is called the homotopy invariance of the de Rham cohomology.

Let $M$ be a smooth manifold and let $(E,D)$ be a $\infty$-local system on $M \times [0,1]$. Our first task is to show that there is an isomorphism of $\infty$-local systems between the restrictions of $(E,D)$ to $M \times \{0\}$ and $M \times \{1\}$. The following two preliminary results will clear our path.

\begin{lemma}\label{lem:4.1}
Suppose that the $\infty$-local system $(E,D)$ is trivialized over $M \times [0,1]$, that is to say, $E = (M \times [0,1])\times V$ for some $\ZZ$-graded vector space $V = \bigoplus_{k \in \ZZ} V^k$ and $D = \uud - \alpha$ for some homogeneous element $\alpha \in\Omega^{\sbullet} (M \times [0,1],\End(V))$ of total degree $1$ satisfying the Maurer-Cartan equation. Then the iterated integral $\Phi^{\alpha}(1)$ of $\alpha$ defines an isomorphism of $\infty$-local systems from $\iota_0^* (E,D)$ onto $\iota_1^* (E,D)$.
\end{lemma}

\begin{proof}
We first show that $\Phi^{\alpha}(1)$ is a morphism of $\infty$-local systems from $\iota_0^* (E,D)$ to $\iota_1^* (E,D)$. To this end, we need to check that $(\uud - \iota_1^* \alpha) \circ \Phi^{\alpha}(1) = \Phi^{\alpha}(1) \circ (\uud - \iota_0^* \alpha)$, or, what is the same,
$$
\ud \Phi^{\alpha}(1) = \iota_1^* \alpha \wedge \Phi^{\alpha}(1) - \Phi^{\alpha}(1) \wedge \iota_0^* \alpha.
$$
Using the formula given in Proposition~\ref{prop:3.6}, we find
\begin{align*}
 \ud \Phi^{\alpha}(1)  = &\sum_{n=1}^{\infty} \sum_{i=1}^{n} (-1)^{n + i -1 + \varepsilon(n)} \int_{\Delta_{n}} \pi_{1}^* \alpha  \wedge \cdots \wedge \pi_{i-1}^* \alpha \wedge  \pi_{i}^*\ud \alpha  \wedge \pi_{i+1}^* \alpha \wedge \cdots \wedge \pi_{n}^* \alpha \\
&+ \sum_{n=2}^{\infty} \sum_{i=1}^{n-1}(-1)^{i + \varepsilon(n)} \int_{\Delta_{n-1}} \pi_{1}^* \alpha  \wedge \cdots \wedge \pi_{i-1}^* \alpha \wedge  \pi_{i}^* (\alpha \wedge \alpha)  \wedge \pi_{i+1}^* \alpha \wedge \cdots \wedge \pi_{n}^* \alpha \\
&+ \sum_{n=1}^{\infty} (-1)^{n + \varepsilon(n)} \left( \int_{\Delta_{n-1}} \pi_{1}^* \alpha  \wedge \cdots \wedge \pi_{n-1}^* \alpha \right) \wedge \iota_0^* \alpha \\
&+ \sum_{n=1}^{\infty} (-1)^{n-1 + \varepsilon(n)} \iota_1^{*} \alpha \wedge \left( \int_{\Delta_{n-1}} \pi_{1}^* \alpha  \wedge \cdots \wedge \pi_{n-1}^* \alpha \right).
\end{align*}
Next, notice that $\varepsilon(n)=\varepsilon(n-1) + n- 1$, from which we obtain $(-1)^{n-1+\varepsilon(n)}=(-1)^{\varepsilon(n-1)}$ and $(-1)^{n+\varepsilon(n)}=(-1)^{\varepsilon(n-1)+1}$. Therefore,
\begin{align*}
 \ud \Phi^{\alpha}(1)  = &\sum_{n=1}^{\infty} \sum_{i=1}^{n} (-1)^{ i + \varepsilon(n-1)} \int_{\Delta_{n}} \pi_{1}^* \alpha  \wedge \cdots \wedge \pi_{i-1}^* \alpha \wedge  \pi_{i}^*\ud \alpha  \wedge \pi_{i+1}^* \alpha \wedge \cdots \wedge \pi_{n}^* \alpha \\
&- \sum_{n=2}^{\infty} \sum_{i=1}^{n-1}(-1)^{i + \varepsilon(n-1)} \int_{\Delta_{n-1}} \pi_{1}^* \alpha  \wedge \cdots \wedge \pi_{i-1}^* \alpha \wedge  \pi_{i}^* (\alpha \wedge \alpha)  \wedge \pi_{i+1}^* \alpha \wedge \cdots \wedge \pi_{n}^* \alpha \\
&- \sum_{n=1}^{\infty} (-1)^{\varepsilon(n-1)}  \left(\int_{\Delta_{n-1}} \pi_{1}^* \alpha  \wedge \cdots \wedge \pi_{n-1}^* \alpha \right) \wedge \iota_0^* \alpha \\
&+ \sum_{n=1}^{\infty} (-1)^{\varepsilon(n-1)} \iota_1^{*} \alpha \wedge \left(  \int_{\Delta_{n-1}} \pi_{1}^* \alpha  \wedge \cdots \wedge \pi_{n-1}^* \alpha \right).
\end{align*}
But $\alpha$ satisfies the Maurer-Cartan equation, so that $\ud \alpha - \alpha \wedge \alpha = 0$. Thus the first and the second term on the right hand side of this equality cancel out, leaving us with
\begin{align*}
 \ud \Phi^{\alpha}(1)  &=\iota_1^{*} \alpha \wedge \left( \sum_{n=1}^{\infty} (-1)^{\varepsilon(n-1)}   \int_{\Delta_{n-1}} \pi_{1}^* \alpha  \wedge \cdots \wedge \pi_{n-1}^* \alpha \right) \\
 &\phantom{=}\quad\quad - \left(\sum_{n=1}^{\infty} (-1)^{\varepsilon(n-1)} \int_{\Delta_{n-1}} \pi_{1}^* \alpha  \wedge \cdots \wedge \pi_{n-1}^* \alpha \right) \wedge \iota_0^* \alpha \\
 &=  \iota_1^{*} \alpha \wedge \Phi^{\alpha}(1) - \Phi^{\alpha}(1) \wedge \iota_0^* \alpha,
\end{align*}
as desired.

It remains to show that $\Phi^{\alpha}(1)$ is invertible. For this, we decompose $\Phi^{\alpha}(1)= \sum_{p \geq 0} \Phi^{\alpha}_p(1)$, where $\Phi^{\alpha}_p(1)$ is of partial degree $p$ with respect to the total $\ZZ$-grading on $\Omega^{\sbullet}(M, \End(V))$. From the definition, it is plain to see that $\Phi_0^{\alpha}(1) = \Phi^{\alpha_1}(1)$, with $\alpha_1$ being the component of $\alpha$ of partial degree $1$ with respect to the total $\ZZ$-grading on $\Omega^{\sbullet}(M \times [0,1], \End(V))$. Since the latter defines a connection which preserves the $\ZZ$-grading, $\Phi_0^{\alpha}(1)$ is nothing but the parallel transport with respect to such connection along the paths $t \mapsto M \times \{t\}$. Thus, as is well known, $\Phi_0^{\alpha}(1)$ is invertible. Therefore $\Phi^{\alpha}(1)$ is the sum of an invertible element and a nilpotent element. From this it follows easily that $\Phi^{\alpha}(1)$ must necessarily be invertible.
\end{proof}

\begin{lemma}\label{lem:4.2}
Suppose that the $\infty$-local system $(E,D)$ is trivialized over $M \times [0,1]$ as both $E = (M \times [0,1])\times V$ and $E=(M \times [0,1]) \times W$ for some $\ZZ$-graded vector spaces $V = \bigoplus_{k \in \ZZ}V^k$ and $W = \bigoplus_{k \in \ZZ} W^k$ with $D= \uud - \alpha$ and $D = \uud - \beta$ for some homogenous elements $\alpha \in \Omega^{\sbullet}(M \times [0,1],\End(V))$ and $\beta \in \Omega^{\sbullet}(M \times [0,1],\End(W))$ of total degree $1$ satisfying the Maurer-Cartan equation. Let $g \in \Omega^0(M \times [0,1],\Hom(V,W))$ be a transition isomorphism associated with these trivializations. Then
$$
\Phi^{\alpha}(1) = (\iota_1^* g)^{-1} \Phi^{\beta}(1) \iota_0^* g.
$$
\end{lemma}

\begin{proof}
As explained in Section~\ref{subsec:2.2}, the transformation rule from $\beta$ to $\alpha$ can be expressed as $\alpha= g^{-1}\beta g - g^{-1} \ud g$. Hence the desired conclusion follows by applying Proposition~\ref{prop:3.3} with $t=1$.
\end{proof}

As an easy application of the above, we derive the following.

\begin{proposition}\label{prop:4.3}
Let $(E,D)$ be an $\infty$-local system over $M \times [0,1]$. Then there is an isomorphism of $\infty$-local systems from $\iota_0^* (E,D)$ onto $\iota_1^* (E,D)$.
\end{proposition}

\begin{proof}
Suppose that the $\infty$-local system $(E,D)$ is trivialized on an open covering $\{U_i\}$ of $M$. Thus $E \vert_{U_i} = (U_i \times [0,1]) \times V_i$ for some $\ZZ$-graded vector space $V_i = \bigoplus_{k \in \ZZ}V_i^k$ and $D \vert_{U_i} = \uud - \alpha_i$ for some homogeneous element $\alpha_i \in \Omega^{\sbullet}(U_i \times [0,1], \End(V_i))$ of total degree $1$ satisfying the Maurer-Cartan equation. By virtue of Lemma~\ref{lem:4.1}, the iterated integral $\Phi^{\alpha_i}(1) \in \Omega^{\sbullet}(U_i ,\End(V_i))$ determines an isomorphism of $\infty$-local systems from $\iota_0^*(E\vert_{U_i}, D\vert_{U_i})$ onto $\iota_1^*(E\vert_{U_i}, D\vert_{U_i})$. On the other hand, if $U_i \cap U_j \neq \varnothing$ and $g_{ji} \in \Omega^{0}((U_i \cap U_j) \times [0,1], \Hom(V_i,V_j))$ is the corresponding transition isomorphism, then, by Lemma~\ref{lem:4.2}, we have
$$
\Phi^{\alpha_i}(1) = (\iota_1^* g_{ji})^{-1} \Phi^{\alpha_j}(1) \iota_0^* g_{ji}.
$$
Hence, the various $\Phi^{\alpha_i}(1)$ piece together to yield a well-defined element $\Phi \in \Omega^{\sbullet}(M,\Hom(\iota_0^*E,\iota_1^*E))$, which is an isomorphism of $\infty$-local systems from $\iota_0^* (E,D)$ onto $\iota_1^* (E,D)$.
\end{proof}

From this result it would seem reasonable to assume that there is a natural isomorphism between the pullback functors $\iota_0^*,\iota_1^*\colon \Loc_{\infty} (M \times [0,1]) \to \Loc_{\infty} (M)$. However, with a careful juggling we can show that this is not the case. In fact, something more precise holds.

\begin{proposition}\label{prop:4.4}
There exists an $\A_{\infty}$-natural isomorphism $\lambda \colon \iota_0^* \Rightarrow \iota_1^*$.
\end{proposition}

\begin{proof}
For every $\infty$-local system $(E,D)$ on $M \times [0,1]$, we let $\lambda_0 (E,D) \in \uZ^0 \Omega^{\sbullet}(M,\Hom(\iota_0^*E,\iota_1^* E))$ be the isomorphism given by Proposition~\ref{prop:4.3}. We need to define linear maps of degree $-n$
\begin{gather*}
\begin{align*}
\lambda_n \colon \Omega^{\sbullet}(M\times[0,1],\Hom(E_{n-1},E_n)) \otimes \cdots \otimes \Omega^{\sbullet}(M\times[0,1],\Hom(E_{0},E_1))  \to \Omega^{\sbullet}(M,\Hom(\iota_0^*E_0,\iota_1^* E_n))
\end{align*}
\end{gather*}
for every collection of $\infty$-local systems $(E_0,D_0),\dots,(E_n,D_n)$ on $M \times [0,1]$. As with the argument that lead to the proof of Proposition~\ref{prop:4.3}, we construct such maps by assuming first that each $\infty$-local system $(E_i,D_i)$ is trivialized over $M \times [0,1]$, that is, $E_i = (M \times [0,1]) \times V_i$ for some $\ZZ$-graded vector space $V_i = \bigoplus_{k \in \ZZ} V_i^k$ and $D_i = \uud - \alpha_i$ for some homogeneous element $\alpha_i \in \Omega^{\sbullet}(M \times [0,1], \End(V_i))$ of total degree $1$ satisfying the Maurer-Cartan equation. On that account, let us take homogeneous elements $\xi_0 \in \Omega^{\sbullet}(M\times[0,1],\Hom(V_{0},V_1)), \dots, \xi_{n-1} \in \Omega^{\sbullet}(M\times[0,1],\Hom(V_{n-1},V_n))$. If we put $V = \bigoplus_{i=0}^{n} V_i$ then both the elements $\alpha_0,\dots, \alpha_n$ and the elements $\xi_0, \dots, \xi_{n-1}$ may be seen as elements of $\Omega^{\sbullet}(M \times [0,1],\End(V))$. Thus, if we set $\omega = \sum_{i=0}^{n}\alpha_i + \sum_{i=0}^{n-1}\xi_i$, the corresponding iterated integral $\Phi^{\omega}(1)$ determines an element of $\Omega^{\sbullet}(M,\End(V))$. With this understood, we define $\lambda_n^{\{\alpha_v\}}(\xi_{n-1} \otimes \dots \otimes \xi_0) \in \Omega^{\sbullet}(M,\Hom(V_0,V_n))$ to be the $(0,n)$ block entry of $\Phi^{\omega}(1)$, and it is straightforward to verify that this prescription determines a linear map
\begin{gather*}
\begin{align*}
\lambda_n^{\{\alpha_v\}} \colon \Omega^{\sbullet}(M\times[0,1],\Hom(V_{n-1},V_n)) \otimes \cdots \otimes \Omega^{\sbullet}(M\times[0,1],\Hom(V_{0},V_1))  \to \Omega^{\sbullet}(M,\Hom(V_0,V_n)).
\end{align*}
\end{gather*}
We must check that this map has degree $-n$. To this end, notice that $\lambda_n^{\{\alpha_v\}}(\xi_{n-1} \otimes \dots \otimes \xi_0)$ is an infinite sum of terms that contain integrals of the form
\begin{gather*}
\begin{align*}
\int_{\Delta_q} \Bigg( \bigwedge_{j=1}^{i_{n-1}-1} \pi_j^*\alpha_{n}\Bigg) \wedge \pi_{i_{n-1}}^* \xi_{n-1} \wedge \Bigg( \bigwedge_{j=i_{n-1}+1}^{i_{n-2} - 1} \pi_j^*\alpha_{n-1}\Bigg) \wedge \cdots \wedge \Bigg( \bigwedge_{j=i_1+1}^{i_0 - 1} \pi_j^*\alpha_1\Bigg) \wedge \pi_{i_0}^* \xi_0 \wedge \Bigg( \bigwedge_{j=i_0+1}^{q} \pi_j^*\alpha_0\Bigg)
\end{align*}
\end{gather*}
where $q \geq n$  and $1 \leq i_{n-1} < i_{n-2} < \cdots < i_{1} < i_{0} \leq q$. Since the total degree of each of the integrands above is $\sum_{i=0}^{n-1}\vert \xi_i \vert + q - n$, it follows that the total degree of these terms is $\sum_{i=0}^{n-1}\vert \xi_i \vert - n$. From this we conclude that the map $\lambda_n^{\{\alpha_v\}}$ has degree $-n$.

Next, we must show that the linear maps $\lambda_n^{\{\alpha_v\}}$ satisfy the required relations to be an $\A_{\infty}$-natural transformation. As indicated in Section~\ref{subsec:2.1}, these relations read
\begin{align*}
\iota_1^* \xi_{n-1} \wedge \lambda_{n-1}^{\{\alpha_v\}} (\xi_{n-2} \otimes \cdots \otimes \xi_{0}) - &(-1)^{\sum_{j=1}^{n-1}\vert \xi_j \vert - n +1} \lambda_{n-1}^{\{\alpha_v\}} (\xi_{n-1} \otimes  \cdots \otimes \xi_{1}) \wedge \iota_0^*\xi_0 \\
&\quad = \lambda^{\{\alpha_v\}} ( b (\xi_{n-1} \otimes \cdots \otimes \xi_0)) + \partial_{D_0,D_n} \lambda_n^{\{\alpha_v\}}(\xi_{n-1} \otimes \dots \otimes \xi_0),
\end{align*}
for every composable chain of homogenous elements $\xi_0 \in \Omega^{\sbullet}(M \times [0,1],\Hom(V_0,V_1)),\dots, \xi_{n-1}  \in \Omega^{\sbullet}(M \times [0,1],\Hom(V_{n-1},V_{n}))$, where, bearing in mind the notations and definitions of Section~\ref{subsec:2.2},
\begin{gather*}
\begin{align*}
b (\xi_{n-1} \otimes \cdots \otimes \xi_0) =& \sum_{i=0}^{n-1}(-1)^{\sum_{j=i+1}^{n-1}\vert \xi_{j}\vert + n - i -1} \xi_{n-1} \otimes \cdots \otimes \xi_{i+1} \otimes \ud \xi_{i} \otimes \xi_{i-1} \otimes \cdots \otimes \xi_0 \\
& - \sum_{i=0}^{n-1}(-1)^{\sum_{j=i+1}^{n-1}\vert \xi_{j}\vert + n - i -1} \xi_{n-1} \otimes \cdots \otimes \xi_{i+1} \otimes (\alpha_{i+1} \wedge \xi_i) \otimes \xi_{i-1} \otimes \cdots \otimes \xi_0 \\
& + \sum_{i=0}^{n-1}(-1)^{\sum_{j=i}^{n-1}\vert \xi_{j}\vert  + n - i -1} \xi_{n-1} \otimes \cdots \otimes \xi_{i+1} \otimes (\xi_{i} \wedge \alpha_i) \otimes \xi_{i-1} \otimes \cdots \otimes \xi_0 \\
&+  \sum_{i=1}^{n-1}(-1)^{\sum_{j=i}^{n-1} \vert \xi_{j}\vert  + n - i -1} \xi_{n-1} \otimes \cdots \otimes \xi_{i+1} \otimes (\xi_{i} \wedge \xi_{i-1}) \otimes \xi_{i-2} \otimes \cdots \otimes \xi_0,
\end{align*}
\end{gather*}
and
\begin{align*}
\partial_{D_0,D_n} \lambda_n^{\{\alpha_v\}}(\xi_{n-1} \otimes \dots \otimes \xi_0) =&\, \ud \lambda_n^{\{\alpha_v\}}(\xi_{n-1} \otimes \dots \otimes \xi_0) -\iota_1^* \alpha_n \wedge \lambda_n^{\{\alpha_v\}}(\xi_{n-1} \otimes \dots \otimes \xi_0) \\
& + (-1)^{\sum_{j=0}^{n-1} \vert \xi_j \vert - n}  \lambda_n^{\{\alpha_v\}}(\xi_{n-1} \otimes \dots \otimes \xi_0) \wedge \iota_0^* \alpha_0.
\end{align*}
Therefore, we are led to verify that
\begin{align*}
&\ud \lambda_n^{\{\alpha_v\}}(\xi_{n-1} \otimes \dots \otimes \xi_0) \\
 &\quad =  \iota_1^* \alpha_n \wedge \lambda_n^{\{\alpha_v\}}(\xi_{n-1} \otimes \dots \otimes \xi_0) - (-1)^{\sum_{j=0}^{n-1} \vert \xi_j \vert - n}  \lambda_n^{\{\alpha_v\}}(\xi_{n-1} \otimes \dots \otimes \xi_0) \wedge \iota_0^* \alpha_0 \\
 &\quad \phantom{=}\,  + \iota_1^* \xi_{n-1} \wedge \lambda_{n-1}^{\{\alpha_v\}} (\xi_{n-2} \otimes \cdots \otimes \xi_{0}) - (-1)^{\sum_{j=1}^{n-1}\vert \xi_j \vert - n +1} \lambda_{n-1}^{\{\alpha_v\}} (\xi_{n-1} \otimes  \cdots \otimes \xi_{1}) \wedge \iota_0^*\xi_0 \\
&\quad \phantom{=}\, -\lambda^{\{\alpha_v\}} ( b (\xi_{n-1} \otimes \cdots \otimes \xi_0)),
\end{align*}
for homogenous elements $\xi_0 \in \Omega^{\sbullet}(M \times [0,1],\Hom(V_0,V_1)),\dots, \xi_{n-1}  \in \Omega^{\sbullet}(M \times [0,1],\Hom(V_{n-1},V_{n}))$. We relegate the proof of this identity to Appendix~\ref{app:A}.

We now examine what happens to the above construction when we change the trivialization of each $\infty$-local system $(E_i,D_i)$. So assume that also $E_i = (M \times [0,1]) \times W_i$ for some $\ZZ$-graded vector space $W_i = \bigoplus_{k \in \ZZ} W_i^{k}$ and $D_i = \uud - \beta_i$ for some homogeneous element $\beta_i \in \Omega^{\sbullet}(M \times [0,1],\End(W_i))$ of total degree $1$ satisfying the Maurer-Cartan equation. We let $g_i \in \Omega^0 (M \times [0,1],\Hom(V_i,W_i))$ denote the associated transition isomorphisms, so that the transformation rule from $\beta_i$ to $\alpha_i$ is expressed as $\alpha_i = g_i^{-1} \beta_i g_i - g_i^{-1} \ud g_i$. Consider a composable chain of homogenous elements $\xi_0 \in \Omega^{\sbullet}(M \times [0,1],\Hom(V_0,V_1)),\dots, \xi_{n-1}  \in \Omega^{\sbullet}(M \times [0,1],\Hom(V_{n-1},V_{n}))$ and set up another composable chain of homogenous elements $\zeta_0 \in \Omega^{\sbullet}(M \times [0,1],\Hom(W_0,W_1)),\dots, \zeta_{n-1}  \in \Omega^{\sbullet}(M \times [0,1],\Hom(W_{n-1},W_{n}))$ by means of the formula $\xi_i = g_{i+1}^{-1} \zeta_i g_i$ for  $i = 0,\dots,n-1$. Similarly as above, we put $W = \bigoplus_{i=0}^{n} W_i$, $\eta= \sum_{i=0}^{n}\beta_i + \sum_{i=0}^{n-1} \zeta_i$ and define $\lambda_n^{\{\beta_v\}}(\zeta_{n-1}\otimes \cdots \otimes \zeta_0) \in \Omega^{\sbullet}(M,\Hom(W_0,W_n))$ to be the $(0,n)$ entry of the iterated integral $\Phi^{\eta}(1) \in \Omega^{\sbullet}(M,\End(W))$. We claim that the following relation holds:
$$
\lambda_n^{\{\alpha_v\}}(\xi_{n-1}\otimes \cdots \otimes \xi_0) = (\iota_1^* g_n)^{-1} \lambda_n^{\{\beta_v\}}(\zeta_{n-1}\otimes \cdots \otimes \zeta_0) \iota_0^* g_0.
$$
To substantiate the claim, we set $g = \sum_{i=0}^{n} g_i \in \Omega^0(M \times [0,1], \Hom(V,W))$ and observe that, by definition, $\omega= g^{-1}\eta g - g^{-1} \ud g$. Hence, by applying Proposition~\ref{prop:3.3} with $t=1$, we have that
$$
\Phi^{\omega}(1) = (\iota_1^* g)^{-1} \Phi^{\eta}(1) \iota_0^* g.
$$
By taking the $(0,n)$ entry to both sides of this equality, we get the desired relation.

Finally, to deal with the general situation, the foregoing argument shows that, just as with the proof of Proposition~\ref{prop:4.3}, the linear maps $\lambda_n$ may be defined by piecing together linear maps defined locally on a trivialising cover for the $(E_i,D_i)$.
\end{proof}

Finally, we can state and prove our main result.

\begin{theorem}\label{thm:4.5}
Let $M, N$ be smooth manifolds and $h$ be a smooth homotopy between maps $f,g \colon M \to N$. Then there exists an $\A_{\infty}$-natural isomorphism $ \mathrm{hol}\colon f^* \Rightarrow g^*$ between the pullback functors $f^*,g^* \colon \Loc_{\infty}(N) \to \Loc_{\infty}(M).$ Such $\A_{\infty}$-natural isomorphism depends only on $h$ and is given explicitly by Chen's iterated integrals.
\end{theorem}

\begin{proof}
By hypothesis, there is a smooth homotopy $h \colon M \times [0,1] \to N$ with $h(x,0)=f(x)$ and $h(x,1)=g(x)$ for all $x \in M$. So we have $h \circ \iota_0 = f$ and $h \circ \iota_1 = g$, and thus $f^* = \iota_0^* \circ h^*$ and $g^* = \iota_1^* \circ h^*$. On the other hand, by Proposition~\ref{prop:4.4}, there is an $\A_{\infty}$-natural isomorphism $\lambda \colon \iota_0^* \Rightarrow \iota_1^*$. By restricting the latter to the full DG subcategory of $\Loc_{\infty}(M \times [0,1])$ consisting of objects of the form $h^*(E,D)$ with $(E,D)$ an object of $\Loc_{\infty}(N)$, we obtain an $\A_{\infty}$-natural isomorphism between $f^*,g^* \colon \Loc_{\infty}(N) \to \Loc_{\infty}(M)$, as wished.
\end{proof}

\begin{remark}
Given a smooth manifold $M$, the usual de Rham map $\phi\colon \Omega^{\sbullet}(M) \rightarrow \mathrm{C}^{\sbullet}(M)$ from differential forms to singular cochains is not an algebra map. However, it does induce an algebra map in cohomology. This curiosity was clarified by Gugenheim \cite{Gugenheim1977}, who extended $\phi$ to an explicit $\mathsf{A}_\infty$-quasi-isomorphism given by certain iterated integrals. This map plays a crucial role in the higher Riemann-Hilbert correspondence \cite{Block-Smith2014}. As explained in  \cite{Abad-Schatz2013}, the higher dimensional holonomies that arise in the higher Riemann-Hilbert correspondence arise via push forward of a Maurer-Cartan element along an $\mathsf{A}_\infty$-morphism constructed from that of Gugenheim. In the particular case where $M$ is an interval $I$, the push forward along the $\mathsf{A}_\infty$-morphism produces the ordinary solutions to the parallel transport equations and realizes an $\mathsf{A}_\infty$-morphism $q\colon \Omega^{\bullet}(I) \rightarrow \mathrm{C}^{\sbullet}(I)$. The differential equations for parallel transport that arise in the computations above are solved explicitly by Chen's iterated integrals. The solutions can also be regarded as arising from the map $q$, which allows one to replace $\mathrm{C}^{\sbullet}(I)$ in the definition of the path algebra by the algebra of differential forms. This can be regarded as the moral reason why the $\mathsf{A}_\infty$-natural transformation defined above arises.
\end{remark}

\section{The categorified Poincar\'e lemma}\label{sec:5}
In this section we derive a series of corollaries of the results obtained in the previous section. Among other things we establish the categorified version of the Poincar\'e lemma for $\infty$-local systems.

Our first corollary is the following.

\begin{corollary}\label{cor:5.1}
If $f \colon M \to N$ is a smooth homotopy equivalence, then the pullback functor $f^{*} \colon \Loc_{\infty}(N) \to \Loc_{\infty}(M)$ is a quasi-equivalence.
\end{corollary}

\begin{proof}
By definition, there exists a smooth map $g \colon N \to M$ such that $f \circ g$ and $g \circ f$ are smoothly homotopic to $\id_N$ and $\id_M$, respectively. Hence, by Theorem~\ref{thm:4.5}, we obtain two $\A_{\infty}$-natural isomorphisms $\lambda \colon g^* \circ f^* \Rightarrow \id_{\Loc_{\infty}(N)}$ and $\mu \colon f^* \circ g^* \Rightarrow \id_{\Loc_{\infty}(M)}$. Therefore, the desired conclusion follows by appealing to Lemma~\ref{lem:2.1}.
\end{proof}

The above has as an immediate consequence the following.

\begin{corollary}[Categorified Poincar\'e lemma]\label{cor:5.2}
If $M$ is contractible, then $\Loc_{\infty}(M)$ is quasi-equivalent to $\DGVect_{\RR}$.
\end{corollary}

\begin{proof}
Since $M$ is contractible, it has the same homotopy type as a point. Thus, according to Corollary~\ref{cor:5.1}, there is a quasi-equivalence between $\Loc_{\infty}(M)$ and $\Loc_{\infty}(\{\mathrm{\ast}\})$. Because $\Loc_{\infty}(\{\mathrm{\ast}\}) = \DGVect_{\RR}$, the desired assertion follows.
\end{proof}

It is now quite easy to see that the following result holds.

\begin{corollary}\label{cor:5.3}
On a contractible manifold $M$, every $\infty$-local system $(E,D)$ is isomorphic to a constant $\infty$-local system.
\end{corollary}

We can strengthen the previous corollary and derive a local normal form for flat superconnections. The following result confirms that locally, all flat superconnections are isomorphic to a constant one.

\begin{corollary} \label{cor:5.4}
For an arbitrary manifold $M$, any $\infty$-local system $(E,D)$ is locally isomorphic to a constant $\infty$-local system, that is, every point has an open neighbourhood $U$ in which $(E\vert_U,D\vert_U)$ is isomorphic to a constant $\infty$-local system.
\end{corollary}

Any contractible open neighbourhood $U$ will obviously work.

\begin{remark}
Some of the results of the paper apply to contexts more general than higher local systems on manifolds, for instance Block's cohesive modules \cite{Block2010} or representations up to homotopy of Lie algebroids \cite{AriasAbad-Crainic2012}. For the sake of concreteness we have decided to remain in the context of $\infty$-local systems.
\end{remark}

\appendix \label{app}
\section{Calculation of $\ud \lambda_n^{\{\alpha_v\}}(\xi_{n-1} \otimes \dots \otimes \xi_0)$}\label{app:A}
In this appendix we embark on the calculation of $\ud \lambda_n^{\{\alpha_v\}}(\xi_{n-1} \otimes \dots \otimes \xi_0)$. The notation and symbols are as in the proof of Proposition~\ref{prop:4.4}.

To begin with, since $\lambda_n^{\{\alpha_v\}}(\xi_{n-1} \otimes \dots \otimes \xi_0)$ is defined as the $(0,n)$ block entry of  iterated integral $\Phi^{\omega}(1)$ of the element $\omega = \sum_{i=0}^n \alpha_i + \sum_{i=0}^{n-1}\xi_i$, it is an infinite sum of terms of the form
$$
\int_0^{1} \int_0^{s_1} \cdots \int_0^{s_{q-1}} \iota_{s_{1}}^* \ui_{\frac{\partial}{\partial s_{1}}} \omega_1 \wedge \iota_{s_{2}}^* \ui_{\frac{\partial}{\partial s_{2}}} \omega_2 \wedge \cdots \wedge \iota_{s_{q}}^* \ui_{\frac{\partial}{\partial s_{q}}} \omega_{q} \, \ud s_{q} \cdots \ud s_2 \ud s_1,
$$
where $q \geq n$ and there is an strictly decreasing $n$-tuple $1 \leq i_{n-1} < i_{n-2} < \cdots < i_1 < i_0 \leq q$ such that $\omega_{i_{k}} = \xi_k$ for $0 \leq k \leq n-1$, $\omega_j = \alpha_n$ for $1 \leq j \leq i_{n-1}-1$, $\omega_j = \alpha_{n-k}$ for $i_{n-k}+1 \leq j \leq i_{n-k-1}-1$ with $k$ running from $1$ to $n-1$, and $\omega_j = \alpha_0$ for $i_0 + 1 \leq j \leq q$.
We want to compute the exterior derivative of each of these terms. For this purpose, by applying an argument similar to that presented in the proof of Proposition~\ref{prop:3.4}, we rewrite the latter as
$$
(-1)^{\sum_{j=1}^{q-1}(q-j)\vert \omega_j\vert}  \int_{\Delta_q} \pi_1^* \omega_1 \wedge \pi_2^*\omega_2 \wedge \cdots \wedge \pi_q^* \omega_q.
$$
Next, we observe that, by virtue of Lemma \ref{lem:3.5}, we have
\begin{align*}
\ud \int_{\Delta_{p}} & \, \pi_{1}^* \omega_1 \wedge \cdots \wedge \pi_{q}^* \omega_q \\
=&  \sum_{i=1}^{q} (-1)^{q + \sum_{j=1}^{i-1}\vert \omega_j \vert} \int_{\Delta_{q}} \pi_{1}^* \omega_1  \wedge \cdots \wedge \pi_{i-1}^* \omega_{i-1} \wedge  \pi_{i}^*\ud \omega_{i}  \wedge \pi_{i+1}^* \omega_{i+1} \wedge \cdots \wedge \pi_{q}^* \omega_q \\
&+  \sum_{i=1}^{q-1}(-1)^{i} \int_{\Delta_{q-1}} \pi_{1}^* \omega_1  \wedge \cdots \wedge \pi_{i-1}^* \omega_{i-1} \wedge  \pi_{i}^* (\omega_{i} \wedge \omega_{i+1})  \wedge \pi_{i+1}^* \omega_{i+2} \wedge \cdots \wedge \pi_{q-1}^* \omega_q \\
&+  (-1)^{q} \left( \int_{\Delta_{q-1}} \pi_{1}^* \omega_1  \wedge \cdots \wedge \pi_{q-1}^* \omega_{q-1} \right) \wedge \iota_0^* \omega_q \\
&+  (-1)^{(q-1) \vert \omega_1\vert} \iota_1^{*} \omega_1 \wedge \left( \int_{\Delta_{q-1}} \pi_{1}^* \omega_2  \wedge \cdots \wedge \pi_{q-1}^* \omega_{q} \right).
\end{align*}
Therefore, by the foregoing, we find that
\begin{align*}
&\ud \int_0^{1} \int_0^{s_1} \cdots \int_0^{s_{q-1}} \iota_{s_{1}}^* \ui_{\frac{\partial}{\partial s_{1}}} \omega_1 \wedge \iota_{s_{2}}^* \ui_{\frac{\partial}{\partial s_{2}}} \omega_2 \wedge \cdots \wedge \iota_{s_{q}}^* \ui_{\frac{\partial}{\partial s_{q}}} \omega_{q} \, \ud s_{q} \cdots \ud s_2 \ud s_1 \\
&\quad = \sum_{i=1}^{q}(-1)^{\sum_{j=1}^{i-1}\vert \omega_j\vert - i} \int_0^{1} \int_0^{s_1} \cdots \int_0^{s_{q-1}} \iota_{s_{1}}^* \ui_{\frac{\partial}{\partial s_{1}}} \omega_1 \wedge \iota_{s_{2}}^* \ui_{\frac{\partial}{\partial s_{2}}} \omega_2 \wedge \cdots \wedge \iota_{s_{i-1}}^* \ui_{\frac{\partial}{\partial s_{i-1}}} \omega_{i-1}    \\
&\qquad\qquad\qquad\qquad\qquad\qquad\wedge  \iota_{s_{i}}^* \ui_{\frac{\partial}{\partial s_{i}}}\ud \omega_{i} \wedge \iota_{s_{i+1}}^* \ui_{\frac{\partial}{\partial s_{i+1}}} \omega_{i+1} \wedge \cdots \wedge \iota_{s_{q}}^* \ui_{\frac{\partial}{\partial s_{q}}} \omega_{q} \, \ud s_{q} \cdots \ud s_2 \ud s_1 \\
&\phantom{\quad =}\,\, + \sum_{i=1}^{q-1}(-1)^{\sum_{j=1}^{i}\vert \omega_j\vert +i}  \int_0^{1} \int_0^{s_1} \cdots \int_0^{s_{q-2}} \iota_{s_{1}}^* \ui_{\frac{\partial}{\partial s_{1}}} \omega_1 \wedge \iota_{s_{2}}^* \ui_{\frac{\partial}{\partial s_{2}}} \omega_2 \wedge \cdots \wedge \iota_{s_{i-1}}^* \ui_{\frac{\partial}{\partial s_{i-1}}} \omega_{i-1}    \\
&\qquad\qquad\qquad\qquad\wedge  \iota_{s_{i}}^* \ui_{\frac{\partial}{\partial s_{i}}}(\omega_i \wedge \omega_{i+1}) \wedge \iota_{s_{i+1}}^* \ui_{\frac{\partial}{\partial s_{i+1}}} \omega_{i+2} \wedge \cdots \wedge \iota_{s_{q-1}}^* \ui_{\frac{\partial}{\partial s_{q-1}}} \omega_{q} \, \ud s_{q-1} \cdots \ud s_2 \ud s_1  \\
&\phantom{\quad =}\,\, + (-1)^{\sum_{j=1}^{q-1}\vert \omega_j\vert + q} \Bigg(\int_0^{1} \int_0^{s_1} \cdots \int_0^{s_{q-2}}  \iota_{s_{1}}^* \ui_{\frac{\partial}{\partial s_{1}}} \omega_1 \wedge \iota_{s_{2}}^* \ui_{\frac{\partial}{\partial s_{2}}} \omega_2 \\
&\qquad\qquad\qquad \qquad\qquad\qquad \qquad\qquad\qquad  \wedge \cdots \wedge \iota_{s_{q-1}}^* \ui_{\frac{\partial}{\partial s_{q-1}}} \omega_{q-1} \, \ud s_{q-1} \cdots \ud s_2 \ud s_1\Bigg) \wedge \iota_0^* \omega_q \\
&\phantom{\quad =}\,\,+ \iota_1^* \omega_1 \wedge \Bigg( \int_0^{1} \int_0^{s_1} \cdots \int_0^{s_{q-2}} \iota_{s_{1}}^* \ui_{\frac{\partial}{\partial s_{1}}} \omega_2 \wedge \iota_{s_{2}}^* \ui_{\frac{\partial}{\partial s_{2}}} \omega_3  \wedge \cdots \wedge \iota_{s_{q-1}}^* \ui_{\frac{\partial}{\partial s_{q-1}}} \omega_{q} \, \ud s_{q-1} \cdots \ud s_2 \ud s_1 \Bigg).
\end{align*}
Let us analyse each of the terms on the right hand side of this relation separately.

Consider the term inside the first sum. We distinguish two cases. First, assume that $\omega_i = \alpha_i$. Then the sign in front is
$$
(-1)^{\sum_{j=1}^{i-1}\vert \omega_j\vert - i} = (-1)^{\sum_{j=i}^{n-1}\vert \xi_j\vert -n +  i -1},
$$
and so the term reduces to
\begin{align}\label{eqn:A.1}
\begin{split}
&(-1)^{\sum_{j=i}^{n-1}\vert \xi_j\vert -n +  i - 1} \int_0^{1} \int_0^{s_1} \cdots \int_0^{s_{q-1}} \iota_{s_{1}}^* \ui_{\frac{\partial}{\partial s_{1}}} \omega_1 \wedge \iota_{s_{2}}^* \ui_{\frac{\partial}{\partial s_{2}}} \omega_2 \wedge \cdots \wedge \iota_{s_{i-1}}^* \ui_{\frac{\partial}{\partial s_{i-1}}} \omega_{i-1}    \\
&\qquad\qquad\qquad\qquad\qquad\wedge  \iota_{s_{i}}^* \ui_{\frac{\partial}{\partial s_{i}}}\ud \alpha_{i} \wedge \iota_{s_{i+1}}^* \ui_{\frac{\partial}{\partial s_{i+1}}} \omega_{i+1} \wedge \cdots \wedge \iota_{s_{q}}^* \ui_{\frac{\partial}{\partial s_{q}}} \omega_{q} \, \ud s_{q} \cdots \ud s_2 \ud s_1.
\end{split}
\end{align}
Second, suppose that $\omega_i = \xi_i$. Then the sign turns out to be
$$
(-1)^{\sum_{j=1}^{i-1}\vert \omega_j\vert - i} = (-1)^{\sum_{j=i+1}^{n-1}\vert \xi_j\vert -n +  i },
$$
and consequently the term becomes
\begin{align*}
&(-1)^{\sum_{j=i+1}^{n-1}\vert \xi_j\vert -n +  i } \int_0^{1} \int_0^{s_1} \cdots \int_0^{s_{q-1}} \iota_{s_{1}}^* \ui_{\frac{\partial}{\partial s_{1}}} \omega_1 \wedge \iota_{s_{2}}^* \ui_{\frac{\partial}{\partial s_{2}}} \omega_2 \wedge \cdots \wedge \iota_{s_{i-1}}^* \ui_{\frac{\partial}{\partial s_{i-1}}} \omega_{i-1}    \\
&\qquad\qquad\qquad\qquad\qquad\qquad\wedge  \iota_{s_{i}}^* \ui_{\frac{\partial}{\partial s_{i}}}\ud \xi_{i} \wedge \iota_{s_{i+1}}^* \ui_{\frac{\partial}{\partial s_{i+1}}} \omega_{i+1} \wedge \cdots \wedge \iota_{s_{q}}^* \ui_{\frac{\partial}{\partial s_{q}}} \omega_{q} \, \ud s_{q} \cdots \ud s_2 \ud s_1.
\end{align*}
If we add these over all $q \geq n$ and all $1 \leq i_{n-1} < i_{n-2} < \cdots < i_1 < i_0 \leq q$, we get
\begin{equation}\label{eqn:A.2}
(-1)^{\sum_{j=i+1}^{n-1}\vert \xi_j\vert -n +  i } \lambda_n^{\{\alpha_v\}}(\xi_{n-1} \otimes \cdots \otimes \xi_{i+1} \otimes \ud \xi_i \otimes \xi_{i-1} \otimes \cdots \otimes \xi_0).
\end{equation}

We now turn to the term inside the second sum. We have four cases to consider. In the first case, $\omega_i = \omega_{i+1} = \alpha_i$. Then the sign in front is
$$
(-1)^{\sum_{j=1}^{i}\vert \omega_j\vert +i} = (-1)^{\sum_{j=i}^{n-1}\vert \xi_j\vert -n + i},
$$
and thus the term becomes
\begin{align}\label{eqn:A.3}
\begin{split}
&(-1)^{\sum_{j=i}^{n-1}\vert \xi_j\vert -n + i}  \int_0^{1} \int_0^{s_1} \cdots \int_0^{s_{q-2}} \iota_{s_{1}}^* \ui_{\frac{\partial}{\partial s_{1}}} \omega_1 \wedge \iota_{s_{2}}^* \ui_{\frac{\partial}{\partial s_{2}}} \omega_2 \wedge \cdots \wedge \iota_{s_{i-1}}^* \ui_{\frac{\partial}{\partial s_{i-1}}} \omega_{i-1}    \\
&\qquad\qquad\qquad\wedge  \iota_{s_{i}}^* \ui_{\frac{\partial}{\partial s_{i}}}(\alpha_i \wedge \alpha_{i+1}) \wedge \iota_{s_{i+1}}^* \ui_{\frac{\partial}{\partial s_{i+1}}} \omega_{i+2} \wedge \cdots \wedge \iota_{s_{q-1}}^* \ui_{\frac{\partial}{\partial s_{q-1}}} \omega_{q} \, \ud s_{q-1} \cdots \ud s_2 \ud s_1.
\end{split}
\end{align}
In the second case, $\omega_i = \alpha_{i+1}$ and $\omega_i = \xi_i$. Then the sign comes out to be
$$
(-1)^{\sum_{j=1}^{i}\vert \omega_j\vert +i} = (-1)^{\sum_{j=i+1}^{n-1}\vert \xi_j\vert -n + i - 1},
$$
and as a result the term becomes
\begin{align*}
&(-1)^{\sum_{j=i+1}^{n-1}\vert \xi_j\vert -n + i - 1} \int_0^{1} \int_0^{s_1} \cdots \int_0^{s_{q-2}} \iota_{s_{1}}^* \ui_{\frac{\partial}{\partial s_{1}}} \omega_1 \wedge \iota_{s_{2}}^* \ui_{\frac{\partial}{\partial s_{2}}} \omega_2 \wedge \cdots \wedge \iota_{s_{i-1}}^* \ui_{\frac{\partial}{\partial s_{i-1}}} \omega_{i-1}    \\
&\qquad\qquad\qquad\wedge  \iota_{s_{i}}^* \ui_{\frac{\partial}{\partial s_{i}}}(\alpha_{i+1} \wedge \xi_{i}) \wedge \iota_{s_{i+1}}^* \ui_{\frac{\partial}{\partial s_{i+1}}} \omega_{i+2} \wedge \cdots \wedge \iota_{s_{q-1}}^* \ui_{\frac{\partial}{\partial s_{q-1}}} \omega_{q} \, \ud s_{q-1} \cdots \ud s_2 \ud s_1.
\end{align*}
Adding these over all $q \geq n$ and all $1 \leq i_{n-1} < i_{n-2} < \cdots < i_1 < i_0 \leq q$, we obtain
\begin{equation}\label{eqn:A.4}
(-1)^{\sum_{j=i+1}^{n-1}\vert \xi_j\vert -n + i - 1} \lambda_n^{\{\alpha_v\}}(\xi_{n-1} \otimes \cdots \otimes \xi_{i+1} \otimes (\alpha_{i+1} \wedge \xi_i) \otimes \xi_{i-1} \otimes \cdots \otimes \xi_0).
\end{equation}
For the third case, we take $\omega_i = \xi_i$ and $\omega_{i+1}=\alpha_i$. Then the sign in front is
$$
(-1)^{\sum_{j=1}^{i}\vert \omega_j\vert +i} = (-1)^{\sum_{j=i}^{n-1}\vert \xi_j\vert -n + i},
$$
and hence the term becomes
\begin{align*}
&(-1)^{\sum_{j=i}^{n-1}\vert \xi_j\vert -n + i} \int_0^{1} \int_0^{s_1} \cdots \int_0^{s_{q-2}} \iota_{s_{1}}^* \ui_{\frac{\partial}{\partial s_{1}}} \omega_1 \wedge \iota_{s_{2}}^* \ui_{\frac{\partial}{\partial s_{2}}} \omega_2 \wedge \cdots \wedge \iota_{s_{i-1}}^* \ui_{\frac{\partial}{\partial s_{i-1}}} \omega_{i-1}    \\
&\qquad\qquad\qquad\wedge  \iota_{s_{i}}^* \ui_{\frac{\partial}{\partial s_{i}}}(\xi_{i} \wedge \alpha_{i}) \wedge \iota_{s_{i+1}}^* \ui_{\frac{\partial}{\partial s_{i+1}}} \omega_{i+2} \wedge \cdots \wedge \iota_{s_{q-1}}^* \ui_{\frac{\partial}{\partial s_{q-1}}} \omega_{q} \, \ud s_{q-1} \cdots \ud s_2 \ud s_1.
\end{align*}
Once again, adding over all $q \geq n$ and all $1 \leq i_{n-1} < i_{n-2} < \cdots < i_1 < i_0 \leq q$, this yields
\begin{equation}\label{eqn:A.5}
(-1)^{\sum_{j=i}^{n-1}\vert \xi_j\vert -n + i} \lambda_n^{\{\alpha_v\}}(\xi_{n-1} \otimes \cdots \otimes \xi_{i+1} \otimes (\xi_{i} \wedge \alpha_{i}) \otimes \xi_{i-1} \otimes \cdots \otimes \xi_0).
\end{equation}
In the fourth and last case, $\omega_i = \xi_i$ and $\omega_{i+1} = \xi_{i-1}$. Then the sign results in
$$
(-1)^{\sum_{j=1}^{i}\vert \omega_j\vert +i} = (-1)^{\sum_{j=i}^{n-1}\vert \xi_j\vert -n + i},
$$
and the term becomes
\begin{align*}
&(-1)^{\sum_{j=i}^{n-1}\vert \xi_j\vert -n + i} \int_0^{1} \int_0^{s_1} \cdots \int_0^{s_{q-2}} \iota_{s_{1}}^* \ui_{\frac{\partial}{\partial s_{1}}} \omega_1 \wedge \iota_{s_{2}}^* \ui_{\frac{\partial}{\partial s_{2}}} \omega_2 \wedge \cdots \wedge \iota_{s_{i-1}}^* \ui_{\frac{\partial}{\partial s_{i-1}}} \omega_{i-1}    \\
&\qquad\qquad\qquad\wedge  \iota_{s_{i}}^* \ui_{\frac{\partial}{\partial s_{i}}}(\xi_{i} \wedge \xi_{i-1}) \wedge \iota_{s_{i+1}}^* \ui_{\frac{\partial}{\partial s_{i+1}}} \omega_{i+2} \wedge \cdots \wedge \iota_{s_{q-1}}^* \ui_{\frac{\partial}{\partial s_{q-1}}} \omega_{q} \, \ud s_{q-1} \cdots \ud s_2 \ud s_1.
\end{align*}
Thus adding over all $q \geq n$ and all $1 \leq i_{n-1} < i_{n-2} < \cdots < i_1 < i_0 \leq q$, we obtain
\begin{equation}\label{eqn:A.6}
(-1)^{\sum_{j=i}^{n-1}\vert \xi_j\vert -n + i} \lambda_n^{\{\alpha_v\}}(\xi_{n-1} \otimes \cdots \otimes \xi_{i+1} \otimes (\xi_{i} \wedge \xi_{i-1}) \otimes \xi_{i-2} \otimes \cdots \otimes \xi_0).
\end{equation}

Finally, let us consider the two remaining terms. Here we identify two cases. First, assume that $\omega_1 = \alpha_n$ and $\omega_q = \alpha_0$. Then the sign in front of the first summand is
$$
(-1)^{\sum_{j=1}^{q-1}\vert \omega_j \vert + q} = (-1)^{\sum_{j=0}^{n-1}\vert \xi_j \vert - n + 1},
$$
and therefore these two terms become
\begin{align*}
& (-1)^{\sum_{j=0}^{n-1}\vert \xi_j \vert - n + 1} \Bigg(\int_0^{1} \int_0^{s_1} \cdots \int_0^{s_{q-2}}  \iota_{s_{1}}^* \ui_{\frac{\partial}{\partial s_{1}}} \omega_1 \wedge \iota_{s_{2}}^* \ui_{\frac{\partial}{\partial s_{2}}} \omega_2 \\
&\qquad\qquad\qquad \qquad\qquad\qquad \qquad\qquad\qquad  \wedge \cdots \wedge \iota_{s_{q-1}}^* \ui_{\frac{\partial}{\partial s_{q-1}}} \omega_{q-1} \, \ud s_{q-1} \cdots \ud s_2 \ud s_1\Bigg) \wedge \iota_0^* \alpha_0 \\
&+ \iota_1^* \alpha_n \wedge \Bigg( \int_0^{1} \int_0^{s_1} \cdots \int_0^{s_{q-2}} \iota_{s_{1}}^* \ui_{\frac{\partial}{\partial s_{1}}} \omega_2 \wedge \iota_{s_{2}}^* \ui_{\frac{\partial}{\partial s_{2}}} \omega_3  \wedge \cdots \wedge \iota_{s_{q-1}}^* \ui_{\frac{\partial}{\partial s_{q-1}}} \omega_{q} \, \ud s_{q-1} \cdots \ud s_2 \ud s_1 \Bigg).
\end{align*}
Thus if we add these over all $q \geq n$ and all $1 \leq i_{n-1} < i_{n-2} < \cdots < i_1 < i_0 \leq q$, we get
\begin{align}\label{eqn:A.7}
\begin{split}
\iota_1^* \alpha_n \wedge \lambda_n^{\{\alpha_v\}} (\xi_{n-1} \otimes \cdots \otimes \xi_0) + (-1)^{\sum_{j=0}^{n-1}\vert \xi_j \vert - n + 1} \lambda_n^{\{\alpha_v\}} (\xi_{n-1} \otimes \cdots \otimes \xi_0) \wedge \iota_0^* \alpha_0.
\end{split}
\end{align}
Second, suppose that $\omega_1 = \xi_{n-1}$ and $\omega_q = \xi_0$. Then the relevant sign is
$$
(-1)^{\sum_{j=1}^{q-1}\vert \omega_j \vert + q} = (-1)^{\sum_{j=1}^{n-1}\vert \xi_j \vert - n},
$$
and so these two terms become
\begin{align*}
& (-1)^{\sum_{j=1}^{n-1}\vert \xi_j \vert - n} \Bigg(\int_0^{1} \int_0^{s_1} \cdots \int_0^{s_{q-2}}  \iota_{s_{1}}^* \ui_{\frac{\partial}{\partial s_{1}}} \omega_1 \wedge \iota_{s_{2}}^* \ui_{\frac{\partial}{\partial s_{2}}} \omega_2 \\
&\qquad\qquad\qquad \qquad\qquad\qquad \qquad\qquad\qquad  \wedge \cdots \wedge \iota_{s_{q-1}}^* \ui_{\frac{\partial}{\partial s_{q-1}}} \omega_{q-1} \, \ud s_{q-1} \cdots \ud s_2 \ud s_1\Bigg) \wedge \iota_0^* \xi_0 \\
&+ \iota_1^* \xi_{n-1} \wedge \Bigg( \int_0^{1} \int_0^{s_1} \cdots \int_0^{s_{q-2}} \iota_{s_{1}}^* \ui_{\frac{\partial}{\partial s_{1}}} \omega_2 \wedge \iota_{s_{2}}^* \ui_{\frac{\partial}{\partial s_{2}}} \omega_3  \wedge \cdots \wedge \iota_{s_{q-1}}^* \ui_{\frac{\partial}{\partial s_{q-1}}} \omega_{q} \, \ud s_{q-1} \cdots \ud s_2 \ud s_1 \Bigg).
\end{align*}
Adding these over all $q \geq n$ and all $1 \leq i_{n-1} < i_{n-2} < \cdots < i_1 < i_0 \leq q$, yields
\begin{align}\label{eqn:A.8}
\begin{split}
\iota_1^* \xi_{n-1} \wedge \lambda_{n-1}^{\{\alpha_v\}} (\xi_{n-2} \otimes \cdots \otimes \xi_0) + (-1)^{\sum_{j=1}^{n-1}\vert \xi_j \vert - n } \lambda_{n-1}^{\{\alpha_v\}} (\xi_{n-1} \otimes \cdots \otimes \xi_1) \wedge \iota_0^* \xi_0.
\end{split}
\end{align}

Having paved the way, we are at last in a position to explicitly give the full expression for $\ud \lambda_n^{\{\alpha_v\}}(\xi_{n-1} \otimes \dots \otimes \xi_0)$. The first thing to notice is that, owing to the Maurer-Cartan equation satisfied by each $\alpha_i$, when we add over all $q \geq n$ and all $1 \leq i_{n-1} < i_{n-2} < \cdots < i_1 < i_0 \leq q$, the terms associated with the contributions \eqref{eqn:A.1} and \eqref{eqn:A.3} cancel out. Next, observe that, thanks to the definition of the differential $b$, if we add together the contributions \eqref{eqn:A.2}, \eqref{eqn:A.4}, \eqref{eqn:A.5} and \eqref{eqn:A.6}, the result of exactly $- \lambda^{\{\alpha_v\}} \left(b(\xi_{n-1} \otimes \cdots \otimes \xi_0) \right)$. Taking into account the two remaining contributions \eqref{eqn:A.7} and \eqref{eqn:A.8}, we thus obtain
\begin{align*}
&\ud \lambda_n^{\{\alpha_v\}}(\xi_{n-1} \otimes \dots \otimes \xi_0) \\
 &\quad =  \iota_1^* \alpha_n \wedge \lambda_n^{\{\alpha_v\}}(\xi_{n-1} \otimes \dots \otimes \xi_0) - (-1)^{\sum_{j=0}^{n-1} \vert \xi_j \vert - n}  \lambda_n^{\{\alpha_v\}}(\xi_{n-1} \otimes \dots \otimes \xi_0) \wedge \iota_0^* \alpha_0 \\
 &\quad \phantom{=}\,  + \iota_1^* \xi_{n-1} \wedge \lambda_{n-1}^{\{\alpha_v\}} (\xi_{n-2} \otimes \cdots \otimes \xi_{0}) - (-1)^{\sum_{j=1}^{n-1}\vert \xi_j \vert - n +1} \lambda_{n-1}^{\{\alpha_v\}} (\xi_{n-1} \otimes  \cdots \otimes \xi_{1}) \wedge \iota_0^*\xi_0 \\
&\quad \phantom{=}\, -\lambda^{\{\alpha_v\}} ( b (\xi_{n-1} \otimes \cdots \otimes \xi_0)),
\end{align*}
which is the relation we wanted to verify.

\section{$\A_{\infty}$-categories, $\A_{\infty}$-functors and $\A_{\infty}$-natural transformations}\label{app:B}
In this appendix, we review the basic notions of the theory of $\A_{\infty}$-categories. A full treatment of the subject can be found in \cite{Seidel2008}.

An \emph{non-unital} $\A_{\infty}$-\emph{category} $\Ccal$ over a field $K$ consists of a set of objects $\Ob \Ccal$, a $\ZZ$-graded $K$-vector space $\Hom_{\Ccal}(X,Y)$ for any pair of objects $X,Y \in \Ob \Ccal$, and composition maps of degree~$2$
$$
m_n^{\Ccal} \colon s\!\Hom_{\Ccal} (X_{n-1},X_n) \otimes_K \cdots \otimes_K s\!\Hom_{\Ccal} (X_{0},X_1) \to \Hom_{\Ccal} (X_{0},X_n)
$$
for every collection $X_0,\dots, X_n \in \Ob \Ccal$, such that for all chains of homogeneous elements $a_1 \in s\!\Hom_{\Ccal} (X_{0},X_1), \dots, a_{n} \in s\!\Hom_{\Ccal} (X_{n-1},X_n)$ the $\A_{\infty}$-\emph{associativity equation}
$$
\sum_{i=1}^n\sum_{j=0}^{n-i} (-1)^{\sum_{k=1}^{j} \vert a_k \vert- j} m_{n-i+1}^{\Ccal}(a_{n}\otimes \cdots \otimes a_{i+j+1} \otimes m_{i}^{\Ccal}(a_{i+j}\otimes\cdots\otimes a_{j+1})\otimes a_{j} \otimes \cdots \otimes a_1) =0
$$
is satisfied for any $n \geq 1$.

The first two $\A_{\infty}$-associativity equations say that $m_1^{\Ccal}$ squares to zero and is a derivation with respect to the composition on $\Ccal$ defined via $m_2^{\Ccal}$. One may hence consider the associated \emph{homotopy category} $\mathbf{Ho}(\Ccal)$, with the same objects as $\Ccal$, morphisms spaces the $0$th cohomology group $\mathrm{H}^{0}(\Hom_{\Ccal}(X,Y),m_1^{\Ccal})$, and composition given by
$$
[b] \circ [a] = (-1)^{\vert b \vert} [m_2^{\Ccal}(a \otimes b)].
$$
Along the same lines, in the case that all higher compositions vanish, the third $\A_{\infty}$-associativity equation simply says that $m_2^{\Ccal}$ is associative. We thus find that a non-unital DG category is a special case of a non-unital $\A_{\infty}$-category $\Ccal$
with $m_n^{\Ccal}= 0$ for all $n > 2$.

If $\Ccal$ and $\Dcal$ are non-unital $\A_{\infty}$-categories, a \emph{non-unital $\A_{\infty}$-functor} $F \colon \Ccal  \to \Dcal$ consists of a map $F_0 \colon \Ob \Ccal \to \Ob \Dcal$ and $K$-linear maps of degree $1$
$$
F_n \colon s\!\Hom_{\Ccal}(X_{n-1},X_{n}) \otimes_K \cdots \otimes_K s\!\Hom_{\Ccal}(X_0,X_1) \to \Hom_{\Dcal}(F(X_0),F(X_n))
$$
for every collection $X_0 ,\dots,X_n \in \Ob \Ccal$, such that for all chains of homogeneous elements $a_1 \in s\!\Hom_{\Ccal} (X_{0},X_1), \dots, a_{n} \in s\!\Hom_{\Ccal} (X_{n-1},X_n)$ the polynomial equation
\begin{align*}
\sum_{i=1}^n\sum_{j=0}^{n-i} (-1)^{\sum_{k=1}^{j} \vert a_k \vert- j}  & F_{n-i+1}(a_{n}\otimes \cdots \otimes a_{i+j+1} \otimes m_{i}^{\Ccal}(a_{i+j}\otimes\cdots\otimes a_{j+1})\otimes a_{j} \otimes \cdots \otimes a_1)
 \\
& \!\!\!\!\!\!\!\!\!\!\!\!= \sum_{r \geq 1} \sum_{s_1+\cdots+s_r = n} m_r^{\Dcal} (F_{s_r} (a_{n}\otimes \cdots \otimes a_{n-s_r+1}) \otimes \cdots \otimes F_{s_1}(a_{s_1} \otimes \cdots \otimes a_1))
\end{align*}
is satisfied for any $n \geq 1$. If $\Ccal$, $\Dcal$ and $\Ecal$ are three non-unital $\A_{\infty}$-categories, and $F \colon \Ccal \to \Dcal$ and $G \colon \Dcal \to \Ecal$ are two non-unital $\A_{\infty}$-functors, then $F$ and $G$ can be composed as follows:
$$
(G \circ F)_n (a_{n} \otimes \cdots \otimes a_1) = \sum_{r \geq 1} \sum_{s_1+\cdots+s_r = n} G_r (F_{s_r}(a_n \otimes\cdots \otimes a_{n-s_r +1} ) \otimes \cdots \otimes F_{s_1}(a_{s_1} \otimes \cdots \otimes a_1)).
$$
Composition is strictly associative, and the identity functor is a neutral element. Any non-unital $\A_{\infty}$-functor $F \colon \Ccal \to \Dcal$ induces an ordinary non-unital functor
$$
\mathbf{Ho}(F) \colon \mathbf{Ho}(\Ccal) \to \mathbf{Ho}(\Dcal)
$$
between the corresponding homotopy categories, acting in the same way on objects and on morphisms by $\mathbf{Ho}(F)([a]) = (-1)^{\vert a \vert} [F_1(a)]$.

Non-unital $\A_{\infty}$-functors between two non-unital $\A_{\infty}$-categories $\Ccal$ and $\Dcal$ can be naturally organized into a non-unital $\A_{\infty}$-category $\EuScript{P} = \AFun(\Ccal,\Dcal)$. An element $\lambda \in \Hom_{\EuScript{P} }^d(F,G)$ of the $\ZZ$-graded $K$-vector space in this $\A_{\infty}$-category is a family of $K$-linear maps of degree $d$
$$
\lambda_n \colon s\!\Hom_{\Ccal}(X_{n-1},X_n) \otimes_K \cdots \otimes_K s\! \Hom_{\Ccal}(X_0,X_1) \to \Hom_{\Dcal}(F(X_0),G(X_n))
$$
for all $X_0,\dots,X_n \in \Ob \Ccal$. In particular, $\lambda_0$ is a family of elements in $\Hom_{\Dcal}^d(F(X),G(X))$ for each object $X \in \Ob \Ccal$. We call such $\lambda$ a $\A_{\infty}$-\emph{pre-natural transformation} of degree $d$ from $F$ to $G$. The derivation $m_1^{\EuScript{P}}$ is given by
\begin{gather*}
\begin{align*}
[m_1^{\EuScript{P}}&(\lambda)]_n (a_n \otimes \cdots \otimes a_1)  = \sum_{r \geq 1} \sum_{i=1}^{r} \sum_{s_1 + \cdots + s_r = n} (-1)^{(d-1)(\sum_{k=1}^{s_1+\cdots+s_{i-1}}\vert a_k \vert - \sum_{k=1}^{i-1} s_k)} \\
&  \times m_r^{\Dcal} (  G_{s_r}(a_n \otimes \cdots \otimes a_{n-s_r + 1}) \otimes \cdots \otimes G_{s_{i+1}} (a_{s_1+\cdots + s_{i+1}} \otimes \cdots \otimes a_{s_1+\cdots + s_{i}+1}) \\
& \qquad\quad \otimes \lambda_{s_i} (a_{s_1+\cdots + s_{i}} \otimes \cdots \otimes a_{s_1+\cdots + s_{i-1}+1}) \\
& \qquad\quad  \otimes F_{s_{i-1}}(a_{s_1 + \cdots s_{i-1}} \otimes \cdots \otimes a_{s_1+\cdots + s_{i-2}+1}) \otimes \cdots \otimes F_{s_1}(a_{s_1} \otimes \cdots \otimes a_1)) \\
&- \sum_{i=1}^n\sum_{j=0}^{n-i} (-1)^{\sum_{k=1}^{n-1} \vert a_k \vert- j+d -1}\lambda_{d-i+1}(a_n \otimes \cdots \otimes a_{i+j+1} \otimes m_i^{\Ccal}(a_{i+j} \otimes \cdots \otimes a_{j+1}) \otimes a_j \otimes \cdots \otimes a_1).
\end{align*}
\end{gather*}
The formulae for the $m_n^{\EuScript{P}}$ with $n \geq 2$ follow a much simpler pattern and can be consulted in \cite{Seidel2008}.

With the above understood, an $\A_{\infty}$-\emph{natural transformation} between two non-unital $\A_{\infty}$-functors $F \colon \Ccal \to \Dcal$ and $G \colon \Ccal \to \Dcal$ is a closed $\A_{\infty}$-pre-natural transformation $\lambda \in \Hom_{\EuScript{P}}(F,G)$ of degree $0$. We simply write $\lambda \colon F \Rightarrow G$ to indicate that we have such a transformation.

Given an $\A_{\infty}$-natural transformation $\lambda \colon F  \Rightarrow G$ between non-unital $\A_{\infty}$-functors $F, G \colon \Ccal \to \Dcal$, consider the elements $[\lambda_0(X)] \in \Hom_{\mathbf{Ho}(\Dcal)}(F(),G(X))$ for each $X \in \Ob \Ccal$. These satisfy the natural condition
$$
[\lambda_0(Y)] \circ [F_1(a)] = [G_1(a)] \circ [\lambda_0(X)]
$$
for all $[a] \in \Hom_{\mathbf{Ho}(\Ccal)}(X,Y)$. Hence, they constitute a natural transformation, which we may denote by $\mathbf{Ho}(\lambda)$, between the ordinary non-unital functors $\mathbf{Ho}(F)$ and $\mathbf{Ho}(G)$.

To close we would like to clarify the connection between the previous definition and the one given in Section~\ref{subsec:2.1}. Suppose that $\Ccal$ and $\Dcal$ are non-unital $\A_{\infty}$-categories with $m_n^{\Ccal} = 0$ and $m_n^{\Dcal}=0$ for all $n > 2$. Thus, as already remarked, both $\Ccal$ and $\Dcal$ are non-unital DG categories, where the differential and composition in $\Ccal$ are given by
$d a = (-1)^{\vert a \vert} m_1^{\Ccal}(a)$, and $b \circ a = (-1)^{\vert a \vert}m_2^{\Ccal}(a,b)$, respectively, and similarly for $\Dcal$. If $F \colon \Ccal \to \Dcal$ and $G \colon \Ccal \to \Dcal$ are two non-unital DG functors, then the closeness condition for an $\A_{\infty}$-natural transformation $\lambda \colon F  \Rightarrow G$ is equivalent to the defining relation presented in Section~\ref{subsec:2.1}, with the replacements $f_0 = a_1, \dots, f_{n-1} = a_n$.

\begin{remark}
As suggested by the referee, in the latter context, the definition of an $\A_{\infty}$-natural transformation can be described in terms of the path algebra. For simplicity, we discuss the case of DG algebras. So let us assume that $\Ccal$ and $\Dcal$ are DG algebras and let $F \colon \Ccal \to \Dcal$ and $G \colon \Ccal \to \Dcal$ be DG maps. If we denote by $\Lambda$ the DG algebra of cellular cochains on the interval $I=[0,1]$, the inclusion of the endpoints induces natural maps $\pi_0 \colon \Lambda \to \RR$ and $\pi_1 \colon \Lambda \to \RR$. Moreover, if $\mathbf{B}$ denotes the bar construction functor, the DG maps $F$ and $G$ induce maps of DG coalgebras $\mathbf{B}(F) \colon \mathbf{B}(\Ccal) \to \mathbf{B}(\Dcal)$ and $\mathbf{B}(G) \colon \mathbf{B}(\Ccal) \to \mathbf{B}(\Dcal)$. Then, an alternative definition of an $\A_{\infty}$-natural transformation from $F$ to $G$ is a DG coalgebra map
$$
\eta \colon \mathbf{B}(\Ccal) \to \mathbf{B}(\Dcal \otimes \Lambda),
$$
such that the map $\mathbf{B}(\id_{\Dcal} \otimes \pi_0) \circ \eta$ is equal to $\mathbf{B}(F)$, and the map $\mathbf{B}(\id_{\Dcal} \otimes \pi_1) \circ \eta$ is equal to $\mathbf{B}(G)$. The reason that this definition is equivalent to the previous one is the following. As described in Section~\ref{subsec:2.1}, an $\A_{\infty}$-natural transformation $\lambda \colon F \Rightarrow G$ is given by a sequence $K$-linear maps of degree $0$
$$
\lambda_n \colon (s\Ccal)^{\otimes n} \to \Dcal,
$$
satisfying the defining relations. Given such a sequence of $K$-linear maps, one can construct a DG coalgebra map $\eta \colon \mathbf{B}(\Ccal) \to \mathbf{B}(\Dcal \otimes \Lambda)$ as follows. Since $\mathbf{B}(\Dcal \otimes \Lambda)$ is cofree, $\eta$ is determined by a $K$-linear map
$$
\overline{\eta} \colon \mathbf{B}(\Ccal) \to s(\Dcal \otimes \Lambda).
$$
If we denote by $\langle 0 \rangle,\langle 1 \rangle,\langle 0 ,1 \rangle$ the three natural generators in $\Lambda$ so that
$$
\Dcal \otimes \Lambda = (\Dcal \otimes \langle 0 \rangle) \oplus (\Dcal \otimes \langle 1 \rangle) \oplus (\Dcal \otimes \langle 0 ,1 \rangle),
$$
then the projections of $\overline{\eta}$ onto $s(\Dcal \otimes \langle 0 \rangle)$ and $s(\Dcal \otimes \langle 1 \rangle)$ are determined by $F$ and $G$, respectively, while the maps $\overline{\eta} \colon (s \Ccal)^{\otimes n} \to (\Dcal \otimes \langle 0 ,1 \rangle)$ are determined by the $\lambda_n$. The relations satisfied by the maps $\lambda_n$ guarantee that $\eta$ is indeed a map of DG coalgebras. Reciprocally, it is clear that the maps $\lambda_n$ can be recovered from $\eta$ so that the two definitions are equivalent.
\end{remark}


\end{document}